\documentclass[a4paper,oneside,11pt]{amsart}

\usepackage{geometry}
\usepackage[english]{babel}
\usepackage{amsmath}
\usepackage[utf8]{inputenc}
\usepackage[T1]{fontenc}
 
\usepackage{amsthm}
\usepackage{amssymb}
\usepackage[all]{xy}
\usepackage[,colorlinks=true,citecolor= DarkBlue,linkcolor=Black]{hyperref}
\usepackage[svgnames,table]{xcolor}
\usepackage{graphicx}

\usepackage{pgf,tikz}\usepackage{mathrsfs}\usetikzlibrary{arrows}

\theoremstyle{definition}
\newtheorem{definition}{Definition}[section]

\theoremstyle{plain}
\newtheorem{theorem}{Theorem}
\newtheorem*{theorem*}{Theorem}
\newtheorem{proposition}[definition]{Proposition}
\newtheorem*{proposition*}{Proposition}
\newtheorem{lemma}[definition]{Lemma}
\newtheorem*{lemma*}{Lemma}

\newtheorem*{sublemma*}{Sub-lemma}

\newtheorem{observation}{Observation}

\newtheorem{fact}{Fact}
\newtheorem*{fact*}{Fact}
\newtheorem{claim}{Claim}
\newtheorem*{claim*}{Claim}

\theoremstyle{remark}
\newtheorem{remark}[definition]{Remark}

\newcommand{\C}{\mathbf{C}}
\newcommand{\R}{\mathbf{R}}

\newcommand{\gl}{\mathfrak{gl}}
\renewcommand{\sl}{\mathfrak{sl}}

\renewcommand{\a}{\mathfrak{a}}
\newcommand{\g}{\mathfrak{g}}

\newcommand{\n}{\mathfrak{n}}
\newcommand{\so}{\mathfrak{so}}
\newcommand{\p}{\mathfrak{p}}

\DeclareMathOperator{\Diff}{Diff}
\DeclareMathOperator{\Hom}{Hom}

\DeclareMathOperator{\Ad}{Ad}

\DeclareMathOperator{\Span}{Span}
\DeclareMathOperator{\Aut}{Aut}

\DeclareMathOperator{\Proj}{Proj}
\DeclareMathOperator{\Conv}{Conv}

\DeclareMathOperator{\GL}{GL}
\DeclareMathOperator{\SL}{SL}

\DeclareMathOperator{\PGL}{PGL}

\DeclareMathOperator{\PO}{PO}

\DeclareMathOperator{\Rk}{Rk}
\DeclareMathOperator{\Conf}{Conf}
\DeclareMathOperator{\Tr}{Tr}
\DeclareMathOperator{\id}{id}
\DeclareMathOperator{\diag}{diag}

\DeclareMathOperator{\Supp}{Supp}

\renewcommand{\S}{\mathbf{S}}

\newcommand{\Ein}{\mathbf{Ein}} 
\newcommand{\X}{\mathbf{X}}

\renewcommand{\epsilon}{\varepsilon}
\renewcommand{\geq}{\geqslant}
\renewcommand{\leq}{\leqslant}
\renewcommand{\bar}{\overline}

\newcommand{\gk}{\tilde{\gamma}_k}
\renewcommand{\gl}{\tilde{\gamma}_l}
\newcommand{\D}{\mathfrak{D}}
\newcommand{\Db}{D}
\newcommand{\pein}{\pi_{\X}}

\title[$(G,\X)$-manifolds with lattice actions]{Projective and conformal closed manifolds with a higher-rank lattice action}

\author{Vincent Pecastaing}
\thanks{Partially supported by FNR grants INTER/ANR/15/11211745 and
OPEN/16/11405402.}
\address{Department of Mathematics, University of Luxembourg,
Maison du nombre, 6 avenue de la Fonte, L-4364 Esch-sur-Alzette, Luxembourg}
\email{vincent.pecastaing@uni.lu}
\date{\today}

\begin{document}

\maketitle

\begin{abstract}
We prove global results about actions of cocompact lattices in higher-rank simple Lie groups on closed manifolds endowed with either a projective class of connections or a conformal class of pseudo-Riemannian metrics of signature $(p,q)$, with $\min(p,q) \geq 2$. In the continuity of a recent article \cite{pecastaing_conformal_lattice}, provided that such a structure is locally equivalent to its model $\X$, the main question treated here is the completeness of the associated $(G,\X)$-structure. Because of the similarities between the model spaces of projective geometry and non-Lorentzian conformal geometry, a number of arguments apply in both contexts. We therefore present the proofs in parallel. The conclusion is that in both cases, when the real-rank is maximal, the manifold is globally equivalent to either the model space $\X$ or its double cover.
\end{abstract}

\section{Introduction}
\label{s:intro}

Zimmer's program suggests that actions of lattices in semi-simple Lie groups on closed manifolds have to be closely related to an homogeneous model. We give in this article two geometric results that confirm this principle and are in the continuity of previous investigations for conformal actions \cite{pecastaing_conformal_lattice}.

Let $\Gamma$ be a lattice in a simple Lie group $G$ of real-rank at least $2$. Among all possible ``geometric actions'' $\rho : \Gamma \rightarrow \Diff(M,\mathcal{S})$ on a closed manifold $M$, we are especially interested in those for which the geometric structure $\mathcal{S}$ is \textit{non-unimodular}. This is due to the fact that these structures do not naturally define a finite $\Gamma$-invariant measure, making more difficult the use of celebrated results such as Zimmer's cocycle super-rigidity. The new powerful tools about invariant measures, introduced in \cite{BRHW} and used in \cite{BFH} for proving Zimmer's conjectures, invite us to pay attention to these non-volume preserving dynamics.

Typical such structures are \textit{parabolic Cartan geometries} (\cite{cap_slovak}) which are curved versions of a given flag manifold $G/P$, because on $G/P$ itself, there exists no finite $\Gamma$-invariant measure. We discuss in this article two cases of actions on parabolic geometries: those preserving a projective class $[\nabla]$ of linear connections and those preserving a conformal case $[g]$ of pseudo-Riemannian metrics.

We remind that two linear connections $\nabla$ and $\nabla'$ on a same manifold $M$ are said to be projectively equivalent if they define the same geodesics up to parametrization. For torsion free connections, this means that there exists a $1$-form $\alpha$ such that $\nabla_X' Y= \nabla_X Y + \alpha(X)Y+\alpha(Y)X$ for all vector fields $X,Y$ (\cite{tanaka}). A projective class $[\nabla]$ is an equivalence class of projectively equivalent linear connections, and the projective group $\Proj(M,[\nabla])$ is the group of diffeomorphisms that preserve this class. A projective structure on a manifold $M^n$ is the same as the data of a Cartan geometry on $M$ modeled on the projective space $\X = \R P^n$ (\cite{kobayashi_nagano}).
Two pseudo-Riemannian metrics $g$ and $g'$ on $M$ are said to be conformal if there exists a smooth positive function $\varphi : M \rightarrow \R_{>0}$ such that $g' = \varphi g$. A conformal class $[g]$ is an equivalence class of conformal metrics and the conformal group $\Conf(M,[g])$ is the group of diffeomorphisms preserving this class. When $n = \dim M \geq 3$, a conformal class of signature $(p,q)$ on $M$ is the same as a normalized Cartan geometry on $M$ modeled on $\X = \Ein^{p,q}$, the model space of conformal geometry discussed below in Section \ref{s:preliminaries}.

From \cite{BFH}, we know that if $\Gamma$ is cocompact and the action $\rho : \Gamma \rightarrow \Diff(M)$ has infinite image, then $\Rk_{\R} G \leq n = \dim M$ and that in the limit case $\Rk_{\R}G = n$, the restricted root-system of $\g$ is $A_n$. It moreover follows from \cite{zhang} that $\g$ is not isomorphic to $\sl(n+1,\C)$. Of course, the natural examples in this limit case are the restriction to $\Gamma$ of the projective action of $\SL(n+1,\R)$ on $\S^n$ or $\R P^n$, and conjecturally they are supposed to be the only examples. It is thus natural to start studying curved versions of these models, \textit{i.e.} projective actions $\Gamma \rightarrow \Proj(M^n,[\nabla])$ with $\Rk_{\R}G=n$.

In \cite{pecastaing_conformal_lattice}, we proved that if $\Gamma$ is uniform and has an unbounded conformal action on a closed pseudo-Riemannian manifold $(M,g)$ of signature $(p,q)$, with $p+q \geq 3$, then $\Rk_{\R} G \leq \min(p,q) + 1$ and that $(M,g)$ is conformally flat when $\Rk_{\R} G = \min(p,q)+1$. This means that the conformal class $[g]$ defines a $\Gamma$-invariant atlas of $(\Conf(\Ein^{p,q}),\Ein^{p,q})$-structure on $M$, which we would like to understand. Projective flatness in the case of a projective action in maximal rank can be derived by the same kind of arguments (see Section \ref{s:projective_flatness}).

So, in both projective and conformal cases, if $\X$ denotes the model space and $G_{\X}$ its automorphisms group, it turns out that if $\rho(\Gamma)$ is unbounded, then $\Rk_{\R} G \leq \Rk_{\R} G_{\X}$ and that the structure is \textit{flat} when equality holds. We see here a strong similarity with Theorem 5 of \cite{bader_frances_melnick} where semi-simple Lie groups actions on parabolic closed manifolds are considered. To obtain a similar conclusion for uniform lattices in such groups, the main problem here is thus to understand globally this $\Gamma$-invariant $(G_{\X},\X)$-structure on $M$. Even when $\Gamma$ is large, this problem is interesting notably because its group structure may not be ``visible'' at a local scale, contrarily to the case of a Lie group action which gives rise to a Lie algebra of vector fields.

The model space of conformal geometry of signature $(p,q)$ is $\Ein^{p,q} = (\S^p \times \S^q)/\{\pm \id\}$ endowed with the conformal class $[-g_{\S^p} \oplus g_{\S^q}]$, where $-\id$ acts via the product of the antipodal maps. When $\min(p,q) \neq 1$, the model spaces $\R P^n$ and $\Ein^{p,q}$ have very similar patterns, in particular they both are natural compactifications of an affine space, via affine charts and stereographic projections, and their universal cover is a $2$-sheeted cover.

Because of these similarities, our approach works for both closed $\R P^n$-manifolds and $\Ein^{p,q}$-manifolds, with \textit{non-Lorentzian} signature $\min(p,q) \geq 2$. The Lorentzian model space $\Ein^{1,n-1}$ behaves a bit differently, due to the non-compactness of its universal cover that invalidates arguments used for proving the invectivity of the developing map. We leave its case for further investigations.

\subsection{Main results}
\label{ss:intro:main_results}

Combined with \cite{pecastaing_conformal_lattice}, we obtain the following global conclusions for actions of uniform lattices of maximal real-rank.

\begin{theorem}
\label{thm:projective}
Let $G$ be a connected simple Lie group with finite center and real-rank $n \geq 2$, and let $\Gamma < G$ be a cocompact lattice. Let $(M^n,\nabla)$ be a closed manifold endowed with a linear connection $\nabla$. Let $\rho : \Gamma \rightarrow \Proj(M,[\nabla])$ be a projective action.

If $\rho(\Gamma)$ is infinite, then $(M,[\nabla])$ is projectively equivalent to either $\S^n$ or $\R P^n$ with their standard projective structures. 
\end{theorem}

Thus the action is a group homomorphism into $\SL^{\pm}(n+1,\R)$ or $\PGL(n+1,\R)$ with infinite image. By Margulis' super-rigidity theorem, $\g \simeq \sl(n+1,\R)$ and $\rho$ extends to a locally faithful action of $\tilde{\SL}(n+1,\R)$. This result can be viewed as a projective counter-part of a result of Zeghib \cite{zeghib_tore} on affine, volume-preserving actions of lattices on closed manifolds, in which he improved a result of Goetze \cite{goetze}. See also \cite{zimmer86_connections,feres92}.

For conformal actions, we obtain a similar statement when the real-rank is maximal.

\begin{theorem}
\label{thm:conform}
Let $(M,g)$ be a closed pseudo-Riemannian manifold of signature $(p,q)$, with metric index $\min(p,q) \geq 2$, and let $\Gamma < G$ be a uniform lattice in a simple Lie group of real-rank $\min(p,q)+1$. Let $\rho : \Gamma \rightarrow \Conf(M,g)$ be a conformal action. 

If $\rho(\Gamma)$ is unbounded in $\Conf(M,g)$, then $(M,g)$ is conformally equivalent to $\Ein^{p,q}$ or its double cover $\tilde{\Ein}^{p,q} = (\S^p \times \S^q,[-g_{\S^p} \oplus g_{\S^q}])$.
\end{theorem}

So, we obtain for $\min(p,q) > 1$ the same conclusion as \cite{bader_nevo,frances_zeghib} where the authors considered conformal actions of all of $G$, also in the case $\Rk_{\R} G = \min(p,q) +1$. 

The action $\rho$ is thus a group homorphism $\Gamma \rightarrow \PO(p+1,q+1)$ or $\Gamma \rightarrow O(p+1,q+1)$ whose image is unbounded. Let us say that $p \leq q$. Using Margulis super-rigidity, we deduce that $\g \simeq \so(p+1,k)$, with $p+1 \leq k \leq q+1$ and that the action extends to a Lie group action up to a ``compact noise'': up to passing to a finite cover of $G$ and lifting $\Gamma$ to it, there exists a compact Lie subgroup $K<\Conf(M,g)$, a Lie group homorphism $\bar{\rho} : G \rightarrow \Conf(M,g)$ with finite kernel and such that $K$ centralizes $\bar{\rho}(G)$, and $\rho_K : \Gamma \rightarrow K$ such that $\rho(\gamma) = \bar{\rho}(\gamma) \rho_K(\gamma)$ for all $\gamma \in \Gamma$. 

Thus, this result shows that essentially, there are no more conformal actions of $\Gamma$ than conformal actions of $G$.

\begin{remark}
It is known (\cite{ferrand71,obata71}) that up to conformal equivalence, for any $n \geq 2$, the only compact Riemannian manifold of dimension $n$ with non-compact conformal group is the standard Riemannian sphere $\S^n$, whose conformal group is $\PO(1,n+1)$. Thus, any conformal action of a higher-rank lattice $\Gamma$ on a compact Riemannian manifold factorizes through a compact Lie group.
\end{remark}

\begin{remark}
It has to be noted that if it exists, a global conclusion for conformal actions of rank $2$ uniform lattices on closed Lorentzian manifolds shall be a bit more complicated as it can be seen in the conclusions Theorem 3 of \cite{frances_zeghib} about semi-simple Lie groups actions.
\end{remark}

\begin{remark}
We also note that the assumption $\Rk_{\R} G = \min(p,q)+1$ it this theorem is necessary as, for instance, \textit{all of} $O(p,q)$ acts conformally on the Hopf manifold $(\R^{p,q}\setminus\{0\})/<2\id>$. The later is conformally flat, but the examples of \cite{frances_counter_examples} provide, for $\min(p,q) \geq 3$, examples of closed, non-conformally flat pseudo-Riemannian manifolds of signature $(p,q)$, on which $O(p-2,q-2)$ acts conformally and essentially. 
\end{remark}

\subsection{Structure of the proof: atlas of maximal charts}
\label{ss:intro:structure_of_proof}

Let $n \geq 2$ and $(p,q)$ such that $\min(p,q) \geq 2$. Let $\X$ be either $\R P^n$ or $\Ein^{p,q}$ and let $G_{\X} = \PGL(n+1,\R)$ or $\PO(p+1,q+1)$ accordingly. Let $G$ be a simple Lie group with finite center, and let $\Gamma<G$ be a uniform lattice. We assume $\Rk_{\R}G = n$ if $\X = \R P^n$ and $\Rk_{\R}G=\min(p,q)+1$ if $\X = \Ein^{p,q}$.

The dynamical starting point of our proof is the existence of sequences $(\gamma_k)$ in $\Gamma$ admitting a \textit{uniformly contracting} dynamical behavior, which are used in \cite{pecastaing_conformal_lattice} for obtaining conformal flatness. With no substantially different arguments - and even less efforts -, we can also exhibit such sequences for projective actions of $\Gamma$, and projective flatness similarly follows by considering the associated Cartan connection. We explain this in the last Section \ref{s:projective_flatness}, and start directly working with locally flat projective and conformal closed manifolds.

These sequences $(\gamma_k)$ contract topologically an open set $U$ to a point $x \in U$ and their derivatives are moreover Lyapunov regular with a uniform Lyapunov spectrum, see Section \ref{ss:uniformly_lyapunov}. The idea is to go backward and consider the $\gamma_k^{-1} U$. We show in Proposition \ref{prop:atlas} that at the limit, the sequence $(\gamma_k^{-1} U)$ gives rise to some maximal domain $U_{\infty}$, which is a trivializing open set for the universal cover $\tilde{M} \rightarrow M$ and such that for any lift $\tilde{U}_{\infty}$, the developing map $\tilde{M} \rightarrow \X$ is injective in restriction to $\tilde{U}_{\infty}$ and sends it to an affine chart domain if $\X = \R P^n$ or a Minkowski patch if $\X = \Ein^{p,q}$. We call such domains $U_{\infty}$ \textit{maximal charts}, and Proposition \ref{prop:atlas} shows that any point of $M$ is contained in a maximal chart.

So, once Proposition \ref{prop:atlas} is established, Theorem \ref{thm:projective} and Theorem \ref{thm:conform} will be a consequence of the following.

\begin{theorem}
\label{thm:atlas}
Let $\X = \R P^n$ or $\Ein^{p,q}$, $n \geq 2$, $\min(p,q) \geq 2$. Let $M$ be a compact manifold endowed with a $(G_{\X},\X)$-structure. Assume that any point of $M$ admits a neighborhood which is either projectively equivalent to $\R^n$ or conformally equivalent to $\R^{p,q}$.
Then, $M$ is isomorphic, as a $(G_{\X},\X)$-manifold, to either $\X$ or $\tilde{\X}$.
\end{theorem}

\subsection*{Plan of the article}

After reminding classic definitions of $\Ein^{p,q}$ and some properties of its stereographic projections in Section \ref{s:preliminaries}, we define in Section \ref{s:maximal_charts} maximal charts of $(G_{\X}, \X)$-manifolds and establish useful properties of these charts that will be used later in the proof of the injectivity of the developing map. Section \ref{s:atlas} is devoted to the proof of Proposition \ref{prop:atlas}. Theorem \ref{thm:atlas} is proved in Section \ref{s:injectivity}, which is easily reduced to the proof of the injectivity of the developing map $\D : \tilde{M} \rightarrow \tilde{\X}$ under the assumption of existence of maximal charts at any point. Finally, we give as announced in Section \ref{s:projective_flatness} the proof of projective flatness of $n$-dimensional manifolds $(M,[\nabla])$ admitting a non-trivial projective action of a cocompact lattice of rank $n$.

\subsection*{Convention and notations}

We will note $M$ a closed $n$-dimensional manifold, with $n \geq 2$. When $M$ is endowed with a conformal structure, we assume $n \geq 3$. For signatures $(p,q)$, with $p+q=n$, we fix the convention $p \leq q$. As in the main theorems, $G$ will always denote a non-compact simple Lie group with finite center and real-rank at least $2$, and $\Gamma$ a uniform lattice in $G$. 

\section{Stereographic projections and Minkowski patches of $\Ein^{p,q}$}
\label{s:preliminaries}

We remind the convention $p \leq q$ and the notation $n=p+q \geq 3$. We will quickly restrict to signatures such that $p \geq 2$. Let $(e_0,\ldots,e_{n+1})$ be a basis of $\R^{p+1,q+1}$ in which the quadratic form reads $Q(u) = 2u_0u_{n+1}+\cdots+2u_pu_{q+1}+u_{p+1}^2+\cdots+u_q^2$. By definition, $\Ein^{p,q} \subset \R P^{n+1}$ is the smooth quadric defined by $\{Q=0\}$, and its conformal structure is the one induced by the restriction of $Q$ to the tangent spaces of the isotropic cone $\{Q=0\}$. Its conformal group is then $\Conf(\Ein^{p,q}) = \PO(p+1,q+1)$.

We note $\S^{n+1}$ the standard Euclidean sphere in $\R^{p+1,q+1}$. The Einstein Universe $\Ein^{p,q}$ is doubly covered by $\{Q=0\} \cap \S^{n+1}$, which is diffeomorphic to $\S^p \times \S^q$. Thus, it is its universal cover whenever $p \geq 2$, and when $p = 1$, its universal cover is diffeomorphic to $\R \times \S^{n-1}$. We fix once and for all a universal cover $\pein : \tilde{\Ein}^{p,q} \rightarrow \Ein^{p,q}$. 

The following was initially observed by Liouville in Riemannian signature (see for instance Section 1 of \cite{cahen_kerbrat}).

\begin{theorem*}
Let $U,V \subset \Ein^{p,q}$ be two connected open subsets and $f : U \rightarrow V$ a conformal map. Then, there exists $\phi \in \Conf(\Ein^{p,q})$ such that $f = \phi|_U$.
\end{theorem*}

\subsection{Minkowski patches and stereographic projections}
\label{s:preliminaries:minkowski_patches}

We recall here the pseudo-Riemannian generalization of the standard stereographic projection $\S^n \setminus \{x\} \rightarrow \R^n$. So, we now assume $p \geq 1$.

Let $v \in \R^{p+1,q+1}$ be an isotropic vector and $x = [v] \in \Ein^{p,q}$. The \textbf{Minkowski patch} $M_x$ associated to $x$ is the intersection of $\Ein^{p,q}$ with the affine chart domain $\{[v'] \ : \ B(v,v') \neq 0\}$ where $B(.,.)$ denotes the scalar product of $\R^{p+1,q+1}$. The \textbf{light-cone} $C_x$ of $x$ is the complement of $M_x$ in $\Ein^{p,q}$, \textit{i.e.} $C_x = \{[v'] \in \Ein^{p,q} \ : \ B(v,v') = 0\}$. We will say that $x$ is the vertex of $M_x$ and $C_x$.

The light-cone $C_x$ is a singular projective variety, with singularity $\{x\}$ and $C_x \setminus \{x\}$ is diffeomorphic to $\R \times \Ein^{p-1,q-1}$. The Minkowski patch $M_x$ is an open-dense subset of $\Ein^{p,q}$ conformally equivalent to $\R^{p,q}$. This last statement is easily observed in the coordinates defined above and with $x=o=[1:0:\cdots:0]$:
\begin{equation*}
M_o = \{[-\frac{<v,v>_{p,q}}{2}: v : 1], \ v \in \R^{p,q}\}.
\end{equation*}
Let us note $s_o : M_o \rightarrow \R^{p,q}$ the inverse of the map $v \in \R^{p,q} \mapsto [-\frac{<v,v>_{p,q}}{2}: v : 1] \in M_o$. From Liouville's Theorem, it follows:

\begin{lemma}
An open subset $U \subset \Ein^{p,q}$ conformally equivalent to $\R^{p,q}$ is a Minkowski patch.
\end{lemma}

\begin{definition}
We call \textit{stereographic projection} any conformal diffeomorphism $s : M_x \rightarrow \R^{p,q}$, where $M_x \subset \Ein^{p,q}$ is a Minkowski patch.
\end{definition}

It has to be noted that any stereographic projection $s : M_x \rightarrow \R^{p,q}$ can be uniquely written $s = s_o \circ \phi^{-1}$ where $\phi \in \Conf(\Ein^{p,q})$ is such that $\phi(o) = x$.

\subsubsection{Lifts to $\tilde{\Ein}^{p,q}$}

As any Minkowski patch $M_x \subset \Ein^{p,q}$ is simply connected, it is a trivializing open subset for the universal cover $\pein$, and we define a Minkowski patch in $\tilde{\Ein}^{p,q}$ as being any connected component $M_x'$ of $\pein^{-1}(M_x)$, where $M_x$ is a Minkowski patch in $\Ein^{p,q}$. Similarly:

\begin{lemma}
An open subset $U \subset \tilde{\Ein}^{p,q}$ which is conformal to $\R^{p,q}$ is a Minkowski patch.
\end{lemma}

\begin{definition}
A stereographic projection in $\tilde{\Ein}^{p,q}$ is a conformal diffeomorphism $\tilde{s} : M_x' \rightarrow  \R^{p,q}$ where $M_x'$ is a Minkowski patch.
\end{definition}

Any such $\tilde{s}$ is of the form $s \circ \pein$, where $M_x = \pein(M_x')$ and $s : M_x \rightarrow \R^{p,q}$ is a stereographic projection.

\subsection{Intersections of Minkowski patches}
\label{ss:preliminaries:intersection_minkowski_patches}

\subsubsection{Intersections in $\Ein^{p,q}$}

Let $M_x \subset \Ein^{p,q}$ be a Minkowski patch and $s : M_x \rightarrow \R^{p,q}$ a stereographic projection. Let $\mathcal{C} \subset \R^{p,q}$ denote the isotropic cone. Let $M_y \subset \Ein^{p,q}$ be another Minkowski patch, with $y \neq x$. There are two possible types for $M_x \cap M_y$:
\begin{itemize}
\item either $y \in M_x$, and in this situation $s(M_x \cap M_y) = \R^{p,q} \setminus (s(y) + \mathcal{C})$
\item or $y \notin M_x$, and $s(M_x \cap M_y) = \R^{p,q} \setminus H_y$, where $H_y \subset \R^{p,q}$ is a degenerate affine hyperplane (of course, $H_y$ depends on $s$).
\end{itemize}
It has to be noted that when $p \geq 2$, $M_x \cap M_y$ always has two connected components, whereas in Lorentzian signature, $M_x \cap M_y$ has three connected components in the first case.

\begin{lemma}
\label{lem:trois_intersection_einpq}
Let $M_x,M_y,M_z$ be three Minkowski patches in $\Ein^{p,q}$. If $M_x \cap M_z = M_y \cap M_z$, then $M_x=M_y$.
\end{lemma}

\begin{proof}
If $M_x \cap M_z = M_z$, then $M_x=M_z=M_y$ by Liouville's theorem. So, let us assume that it is not the case. Then, by the previous paragraph, $x \in M_z$ if and only if $y \in M_z$, and in this case $x=y$ because they are sent by any stereographic projection of $M_z$ to the singularity of a same light-cone.

In the other case, it is enough to observe - in suitable homogeneous coordinates - that when $x \notin M_z$, given a stereographic projection $s : M_z \rightarrow \R^{p,q}$, if $\Delta = v + \R.v_0$ is any affine isotropic line contained in the complement of $s(M_x \cap M_z)$, then $s^{-1}(v+tv_0) \rightarrow x$ as $t \to \pm \infty$. This shows that the data of $M_x \cap M_z$ determines $x$ in this situation, and the lemma is proved.
\end{proof}

\subsubsection{Intersections in $\tilde{\Ein}^{p,q}$}
\label{sss:intersection_einpqtilde}

Let $M_1 \subset \tilde{\Ein}^{p,q}$ be a Minkowski patch, and let $s : M_1 \rightarrow \R^{p,q}$ be a stereographic projection. Let $M_2 \subset \tilde{\Ein}^{p,q}$ be another Minkowski patch such that $M_1 \cap M_2 \neq 0$. We note $\bar{s} : \pein(M_1) \rightarrow \R^{p,q}$ the stereographic projection such that $s = \bar{s} \circ \pein$.

\begin{lemma}
$\pein(M_1 \cap M_2)$ is a connected component of $\pein(M_1) \cap \pein(M_2)$.
\end{lemma}

\begin{proof}
Even though this lemma is valid for $\Ein^{1,n-1}$, we only give a proof for non-Lorentzian signatures $\min(p,q) >1$ which is the case discussed in this article. We suppose $M_1 \neq M_2$, otherwise the statement is obvious.

Let us show that $\pein(M_1 \cap M_2)$ is closed in $\pein(M_1) \cap \pein(M_2)$. Let $\bar{x} \in (\pein(M_1) \cap \pein(M_2)) \setminus \pein(M_1 \cap M_2)$. Let $x \in M_1$ be such that $\pein(x) = \bar{x}$. By definition, $x \notin M_2$. Let $U_x \subset M_1$ be a connected neighborhood of $x$ in restriction to which $\pein$ is injective and such that $\pein(U_x) \subset \pein(M_2)$.
Therefore, $U_x \cap M_2 = \emptyset$ because if not, Lemma \ref{lem:assiettes_emboitees} would imply $U_x \subset M_2$, contradicting $x \notin M_2$. By construction, $\pein(U_x) \cap \pein(M_1 \cap M_2) = \emptyset$ and we get as announced that $\pein(M_1 \cap M_2)$ is closed in $\pein(M_1) \cap \pein(M_2)$. Since it is also open, we get that $\pein(M_1 \cap M_2)$ is a union of connected components of $\pein(M_1) \cap \pein(M_2)$. 

Because we assume $\min(p,q) \geq 2$, as observed above, $\pein(M_1) \cap \pein(M_2)$ has two connected components. And since we cannot have $\pein(M_1 \cap M_2) = \pein(M_1) \cap \pein(M_2)$ (otherwise $\pein|_{M_1 \cup M_2}$ would be injective), we get that $\pein(M_1 \cap M_2)$ must be a single connected component.
\end{proof}

Thus, we deduce the following useful observation.

\begin{observation}
\label{obs:intersection}
When $p = \min(p,q) \geq 2$, given two distinct, non-antipodal Minkowski patches $M_1,M_2$ of $\tilde{\Ein}^{p,q}$ and a stereographic projection $s : M_1 \rightarrow \R^{p,q}$, $s(M_1 \cap M_2)$ is an open set of the form
\begin{itemize}
\item $v_0 + U_S$, $v_0 \in \R^{p,q}$
\item or $v_0 + U_T$, $v_0 \in \R^{p,q}$
\item or $\{v \in \R^{p,q} \ : \ b(v,v_0) > \alpha\}$, with $v_0 \in \mathcal{C} \setminus \{0\}$ and $\alpha \in \R$,
\end{itemize}
where we note $b(v,w) = -v_1w_1 - \cdots - v_pw_p + v_{p+1}w_{p+1} + \cdots + v_nw_n$, $q(v) = b(v,v)$, $\mathcal{C} = \{q = 0\}$, $U_S = \{q > 0\}$, $U_T = \{q<0\}$. For $v_0 \in \mathcal{C} \setminus \{0\}$ and $\alpha \in \R$, we will note $H_{v_0,\alpha} = \{v \in \R^{p,q} \ : \ b(v_0,v) > \alpha\}$.
\end{observation}

\begin{definition}
\label{def:intersection_type}
An open subset $U \subset \R^{p,q}$ of the form $v_0 + U_S$, $v_0 + U_T$ for any $v_0 \in \R^{p,q}$ or $H_{v_0,\alpha}$ for $v_0 \in \mathcal{C} \setminus \{0\}$ and $\alpha \in \R$ is said to be of \textbf{intersection type}.
\end{definition}

We will also use the fact that a Minkowski patch is determined by its intersection with another one.

\begin{lemma}
\label{lem:same_intersection}
Let $M_1,M_2 \subset \tilde{\Ein}^{p,q}$ be two Minkowski patches. Then, $\pein(M_1 \cap \iota (M_2))$ is the complement of $\pein(M_1 \cap M_2)$ in $\pein(M_1) \cap \pein(M_2)$. If $M_1,M_2,M_3 \subset \tilde{\Ein}^{p,q}$ are three Minkowski patches and if $M_1 \cap M_3 = M_2 \cap M_3$, then $M_1 = M_2$.
\end{lemma}
                                      
\begin{proof}         
For the first claim, we may assume $M_1 \neq M_2$ and $M_1 \neq \iota(M_2)$, so that $\pein(M_1)\cap\pein(M_2)$ has two connected components. We have seen that $\pein(M_1 \cap \iota(M_2))$ is one of these components, and by injectivity of $\pein|_{M_1}$, it must be the complement of $\pein(M_1\cap M_2)$.
     												                                                                                                  
Let us assume $M_1 \cap M_3 = M_2 \cap M_3$. The conclusion is clear when these intersections are empty, or equal to all of $M_3$. Let us assume that it is not the case. Then, it follows that $\pein(M_1) \cap \pein(M_3)$ and $\pein(M_2) \cap \pein(M_3)$ have a common connected component, namely $\pein(M_1 \cap M_3)$. In all cases, the boundary of this component in $\pein(M_3)$ is the complement of $\pein(M_1) \cap \pein(M_3)$ in $\pein(M_3)$, as well as the complement of $\pein(M_2) \cap \pein(M_3)$ in $\pein(M_3)$. This proves $\pein(M_1) \cap \pein(M_3) = \pein(M_2) \cap \pein(M_3)$. By Lemma \ref{lem:trois_intersection_einpq}, it follows that $\pein(M_1) = \pein(M_2)$. 

So, either $M_1 = M_2$ or $M_1 = \iota (M_2)$. But the last case is not possible because $M_1 \cap M_3 = M_2 \cap M_3$, and it concludes the proof.
\end{proof}

\begin{remark}
\label{rem:trace_sur_lantipodal}
Let $M_1,M_2 \subset \tilde{\Ein}^{p,q}$ be two distinct and non-antipodal Minkwoski patches. Let $s : M_1 \rightarrow \R^{p,q}$ be a stereographic projection, and let $s' = s \circ \iota : \iota(M_1) \rightarrow \R^{p,q}$. Since $s'(\iota(M_1) \cap M_2) = s(M_1 \cap \iota(M_2))$, it follows from the previous lemma that
\begin{itemize}
\item if $s(M_1 \cap M_2) = v+U_S$, then $s'(\iota(M_1) \cap M_2) = v+U_T$, and
\item if $s(M_1 \cap M_2) = H_{v,\alpha}$, then $s'(\iota(M_1) \cap M_2) = H_{-v,-\alpha}$.
\end{itemize}
\end{remark}

\section{Maximal charts on $\tilde{M}$}
\label{s:maximal_charts}

Let $\X$ denote either $\R P^n$ or $\Ein^{p,q}$, with $\min(p,q) \geq 2$, and $G_{\X}$ its automorphisms group. Let $M$ be a compact manifold endowed with a $(G_{\X},\X)$-structure. We fix $\pi : \tilde{M} \rightarrow M$ a universal cover and we pull back the geometric structure of $M$ to $\tilde{M}$. We note $\pi_{\X} : \tilde{\X} \rightarrow \X$ the natural covering.

We choose $(\D,\tilde{\rho})$ a developing pair modeled on $\tilde{\X}$, \textit{i.e.} a (projective or conformal) immersion $\D : \tilde{M} \rightarrow \tilde{\X}$ and a homomorphism $\tilde{\rho} : \Aut(\tilde{M}) \rightarrow \Aut(\tilde{\X})$ such that $\D$ is $\tilde{\rho}$-equivariant. We note $D = \pein \circ \D : \tilde{M} \rightarrow \X$ and $\rho : \Aut(\tilde{M}) \rightarrow \Aut(\X)$ the natural developing pair with model $\X$ associated to $(\D,\tilde{\rho})$. The homomorphism $\rho$ is $\tilde{\rho}$ followed by the natural projection $\Aut(\tilde{\X}) \rightarrow \Aut(\X)$.

\subsection{Definition of maximal charts and classic lemmas}
\label{ss:maximal_charts:definition}

\begin{definition}
We call \textbf{maximal chart} an open subset $V \subset \tilde{M}$ in restriction to which $\pi$ and $\D$ are injective and such that $\D(V) \subset \S^n$ is an hemisphere if $\X = \R P^n$ or $\D(V) \subset \tilde{\Ein}^{p,q}$ is a Minkowski patch if $\X = \Ein^{p,q}$.
\end{definition}

\begin{remark}
Equivalently, a maximal chart is the same as an open set $U \subset M$ which is either projectively equivalent to $\R^n$ or conformally equivalent to $\R^{p,q}$. Any connected component $V$ of $\pi^{-1}(U)$ will be a maximal chart for the previous definition.
\end{remark}

We will use repeatedly the following classic results about local homeomorphisms. They are stated and proved in \cite{barbot}, Section 2.1. Let $M,N$ be two manifolds and $f : M \rightarrow N$ a local homeomorphism.

\begin{lemma}
\label{lem:assiettes_emboitees}
Let $U \subset M$ and $V \subset N$ be two open sets such that $f|_U$ is a homeomorphism onto $V$. If $W \subset M$ is a connected open subset such that $f(W) \subset V$ and $W \cap U \neq \emptyset$, then $W \subset U$.
\end{lemma}

\begin{definition}
A subset $X \subset M$ of a manifold $M$ is said to be \textbf{locally connected relatively to $M$} if any point $x \in \bar{X} $ has a fundamental system of neighborhoods $\mathcal{V}_x$ such that for all $V \in \mathcal{V}_x$, $V \cap X$ is connected.
\end{definition}

Typically, an affine chart domain is not locally connected relatively to $\R P^n$, whereas a hemisphere is locally connected relatively to $\S^n$.

\begin{lemma}
\label{lem:fermeture}
Let $U \subset M$ be an open subset in restriction to which $f$ is injective. If $f(U)$ is locally connected relatively to $N$, then $f$ is injective in restriction to $\bar{U}$.
\end{lemma}

The next result follows easily:

\begin{lemma}
\label{lem:intersection_connexe}
Let $V_1,V_2 \subset M$ be two open subsets such that $V_1 \cap V_2 \neq \emptyset$, $f|_{V_i}$ is injective for $i = 1,2$, and if $U_i = f(V_i)$, such that $U_1 \cap U_2$ is connected. Then, $f(V_1 \cap V_2) = U_1 \cap U_2$. In particular, $f|_{V_1 \cup V_2}$ is injective.
\end{lemma}

\begin{proof}
We note $U= U_1 \cap U_2$ and consider $W = V_1 \cap f^{-1}(U)$. Then, $f|_W$ is injective and $f(W) = U$ is connected. It implies that $W$ is connected. Thus, as $W \cap V_2 \neq \emptyset$ and $f(W) \subset U_2 = f(V_2)$ we get $W \subset V_2$ by Lemma \ref{lem:assiettes_emboitees}, implying $W = V_1 \cap V_2$.

Thus, $(f|_{V_1})^{-1}(U_1\cap U_2) = V_1 \cap V_2$, and if $x \in V_1$ and $y \in V_2$ have same image, then $f(x)=f(y) \in U_1 \cap U_2$, implying $x \in V_1 \cap V_2$, and finally $x=y$ by injectivity of $f|_{V_2}$.
\end{proof}

\subsection{Relative compactness of maximal charts}
\label{ss:maximal_charts:relative_compactness}

\begin{proposition}
\label{prop:relative_compactness}
Assume that $\tilde{M}$ is covered by maximal charts. Then, any maximal chart $V$ is relatively compact in $\tilde{M}$.
\end{proposition}

\begin{remark}
The conclusion is still valid if we only assume $M$ compact, however this statement is enough for the purpose of this article.
\end{remark}

\begin{proof}
We assume to the contrary that $V$ contains a diverging sequence $(x_k)$. By compactness of $M$, there exists a sequence $\gamma_k \in \pi_1(M)$ such that $\gamma_k.x_k \rightarrow x \in \tilde{M}$. Since $x_k$ leaves any compact subset of $\tilde{M}$, we may assume the $\gamma_k$ pairwise distinct.

The fact that $\pi|_{V}$ is injective means that for any $\gamma \in \pi_1(M)$, if $\gamma V \cap V \neq \emptyset$, then $\gamma = \id$. Consequently, the sequence $V_k := \gamma_k V$ is formed of pairwise disjoint open sets. Let $U_k = \D(V_k)$ and let $V_0 \ni x$ be a maximal chart containing $x$, and let $U_0 = \D(V_0)$. We may assume that for all $k$, $\gamma_k x_k \in V_0$, implying that $V_k \cap V_0 \neq \emptyset$.

\begin{lemma}
The subsets $U_k\cap U_0$ are pairwise disjoint.
\end{lemma}

\begin{proof}
We have seen in Section \ref{sss:intersection_einpqtilde} that when they intersect, two Minkwoski patches in $\tilde{\Ein}^{p,q}$ always have connected intersection. Consequently, the same being obvious for two hemispheres of $\S^n$, if $k$ is such that $V_k \cap V_0 \neq \emptyset$, by Lemma \ref{lem:intersection_connexe}, $\D(V_k \cap V_0) = U_k \cap U_0$. The lemma now follows immediately, as the $V_k \cap V_0$ are pairwise disjoint.
\end{proof}

We finally get a contradiction with the following.

\begin{lemma}
\label{lem:finite_cardinality}
Let $H_0 \subset \S^n$ be a hemisphere. A family $(H_i)_{i \in I}$ of hemispheres such that the $H_0 \cap H_i$ are non-empty and pairwise disjoint has cardinality at most $2$.

Let $M_0 \subset \tilde{\Ein}^{p,q}$ be a Minkowski patch. A family $(M_i)_{i\in I}$ of Minkowski patches such that the $M_0 \cap M_i$, $i \in I$, are non-empty and pairwise disjoint has cardinality at most $2$.
\end{lemma}

\begin{proof}
The first part is almost immediate: if $a : H_0 \rightarrow \R^n$ is an affine chart, then $a(H_0 \cap H_i)$ is either $\R^n$ if $H_0 = H_i$ or an open half-space in $\R^n$ if not.

We fix $s_0 : M_0 \rightarrow \R^{p,q}$ a stereographic projection. For all $i \in I$, we note $U_i = s_0(M_0 \cap M_i)$. They form a family of pairwise disjoint open sets of $\R^{p,q}$ and according to Observation \ref{obs:intersection}, for all $i$, $U_i$ is either a half-space with degenerate boundary, or a translate of $U_S$ or $U_T$. We make use of the following elementary considerations.

\begin{fact}
\label{fact:disjoint_open_sets}
If two half-spaces $H_{v_1,\alpha_1}$ and $H_{v_2,\alpha_2}$ are disjoint, then $v_2=-v_1$ and $\alpha_1 \geq -\alpha_2$. 
A half-space $H_{v,\alpha}$ intersects any translate $v'+U_S$ and any translate of $U_T$. 
Two open sets $v_1+U_S$ and $v_2+U_S$ always intersect, as well as two translates of $U_T$. Moreover, $(v_1+U_S)\cap (v_2+U_T) = \emptyset$ if and only if $v_1=v_2$. 
\end{fact}

So, if $U_1,U_2 \subset \R^{p,q}$ are two disjoint open subsets of intersection type, then, either $U_1 = H_{v_1,\alpha_1}$ and $U_2 = H_{v_2,\alpha_2}$ with $v_1$ isotropic, $v_2 = -v_1$ and $\alpha_1 \geq -\alpha_2$, or $U_1 = v+U_S$ and $U_2 = v + U_T$ with $v \in \R^{p,q}$. It is then clear that any third open subset $U_3$ of intersection type cannot be disjoint from $U_1$ and $U_2$.
\end{proof}
Thus, Lemma \ref{lem:finite_cardinality} is proved, contradicting the existence of the sequence $(U_k \cap U_0)_k$, and the proof of Proposition \ref{prop:relative_compactness} is complete.
\end{proof}

\subsection{Thickenings}
\label{ss:thickenings}

A crucial point in the proof of the main results is the following.

\begin{lemma}
\label{lem:thickenings}
If $V \subset \tilde{M}$ is a relatively compact maximal chart, then there is an open neighborhood $V' \supset \bar{V}$ of the closure of $V$ in restriction to which $\D$ is still injective.
\end{lemma}

\begin{proof}
The first step is to prove that $\D$ is injective in restriction to the closure $\bar{V}$. We simply have to verify that Lemma \ref{lem:fermeture} applies. In the projective case, it is immediate that a hemisphere is locally connected relatively to $\S^n$. Let us see that it is also the case for a Minkowski patch $M_0 \subset \tilde{\Ein}^{p,q}$. 

Let $x \in \partial M_0$, and let $\bar{x} = \pein(x) \in \partial \pein(M_0)$. Let $\bar{x}_0$ be the vertex of $\pein(M_0)$ and let $\bar{x}_1 \in \pein(M_0)$ be a point such that $\bar{x} \notin C(\bar{x}_1)$ and let $M_1 \subset \tilde{\Ein}^{p,q}$ be the Minkowski patch that projects to $\Ein^{p,q} \setminus C(\bar{x}_1)$ and that contains $x$. By construction, $\bar{x}_0 \in \pein(M_1)$. Let $x_0 \in p^{-1}(\bar{x_0})$ be the lift such that $x_0 \in M_1$. Finally, let $s_1 : M_1 \rightarrow \R^{p,q}$ be a stereographic projection such that $s_1(x_0)=0$. By Observation \ref{obs:intersection}, $s_1(M_1 \cap M_0)$ is one of the connected components of $\R^{p,q} \setminus \mathcal{C}$, where $\mathcal{C}$ is the isotropic cone, \textit{i.e.} either $U_S$ or $U_T$. Thus, $x$ has a neighborhood $M_1$ with a chart $s_1 : M_1 \rightarrow \R^{p,q}$ in which $M_0$ is sent to one of the above open subsets. The problem being local, it is reduced to the fact that $U_S$ and $U_T$ are locally connected relatively to $\R^{p,q}$, which can be easily observed.
 
Consequently, $\D$ is injective in restriction to $\bar{V}$, which is compact by assumption. Assume to the contrary that $\D$ is not injective in restriction to any neighborhood of $\bar{V}$. Considering a decreasing sequence $\{V_n\}$ such that $\bar{V} \subset V_n$ and $\bar{V} = \cap V_n$, we obtain two sequences $x_n,y_n \in V_n$ such that $\D(x_n)= \D(y_n)$ and $x_n \neq y_n$. By compactness of $\bar{V}$, we may assume that $V_n$ is relatively compact, and up to an extraction, $(x_n) \rightarrow x \in \bar{V}$ and $(y_n) \rightarrow y \in \bar{V}$. Then $\D(x) = \D(y)$, implying $x=y$. Thus, $(x_n)$ and $(y_n)$ converge to a same limit $x$, contradicting the injectivity of $\D$ on a neighborhood of $x$.
\end{proof}
 
\begin{remark}
It has to be noted that $\D(\bar{V}) = \bar{\D(V)}$ by relative compactness of $V$. In particular, given any small enough neighborhood $\mathcal{V} \supset \bar{\D(V)}$, there exists a neighborhood of $\bar{V}$ on which $\D$ is injective and whose image is $\mathcal{V}$. This will be used in Section \ref{s:injectivity}.
\end{remark}

\section{Atlas of maximal charts}
\label{s:atlas}

We still consider a compact $(G_{\X},\X)$-manifold $M$, with universal cover $\pi : \tilde{M} \rightarrow M$. The aim of this section is to establish that in the dynamical context of a lattice action of maximal real-rank, $M$ is covered by maximal charts.

\begin{proposition}
\label{prop:atlas}
Let $\Gamma$ be a cocompact lattice in a connected simple Lie group $G$ with finite center. 
\begin{enumerate}
\item $\X = \R P^n$. Assume that $\Gamma$ acts projectively on $M$, with infinite image, and that $\Rk_{\R} G = n$. Then, any point of $M$ is contained in a maximal chart.
\item $\X = \Ein^{p,q}$. Assume that $\Gamma$ acts conformally on $M$, with unbounded image, and that $\Rk_{\R} G = p+1$. Then, any point of $M$ is contained in a maximal chart.
\end{enumerate}
\end{proposition}

We fix $\Gamma$ and $G$ satisfying the hypothesis of Proposition \ref{prop:atlas}. Lifting all elements of $\Gamma$ to $\tilde{M}$, we obtain a discrete subgroup $\tilde{\Gamma} < \Aut(\tilde{M})$ containing $\pi_1(M)$ as a normal subgroup, and such that $\tilde{\Gamma} / \pi_1(M) \simeq \Gamma$.

\subsection{Uniformly Lyapunov regular data}
\label{ss:uniformly_lyapunov}

This proposition relies on the dynamical phenomenon which is used in Section \ref{s:projective_flatness} for proving projective flatness, as well as in \cite{pecastaing_conformal_lattice} for proving conformal flatness. Namely:

\begin{lemma}
\label{lem:localisation_lyapunov_data}
In any compact, $\Gamma$-invariant subset of $M$, there is a point $x$ such that there exist a sequence $(\gamma_k)$ in $\Gamma$, a sequence of positive numbers $T_k \rightarrow \infty$, and a connected neighborhood $U$ of $x$ such that
\begin{enumerate}
\item $\gamma_k U \rightarrow \{x\}$ for the Hausdorff topology,
\item for all $v \in T_xM \setminus \{0\}$, $\frac{1}{T_k} \log \|D_x \gamma_k v\| \rightarrow -1$.
\end{enumerate}
\end{lemma}

In fact, we know more than this, but it is all what we need here. 

\begin{proof}
We summarize the ideas for the conformal case, which are easily transferable to the projective one, and refer to Section 6 of \cite{pecastaing_conformal_lattice} for more details. If $M^{\alpha} \rightarrow G/\Gamma$ is the suspension bundle and if $K \subset M$ is a compact $\Gamma$-invariant subset, then $K^{\alpha}:= (K\times M) / \Gamma \subset M^{\alpha}$ is $G$-invariant. Let $A<G$ be a Cartan subspace. We pick a finite $A$-invariant, $A$-ergodic measure $\mu$ supported in $K^{\alpha}$ and that projects to the Haar measure of $G/\Gamma$. Super-rigidity of cocycles and the rigidity of the $\Gamma$-invariant geometric structure on $M$ imply that $\mu$ cannot be $G$-invariant (see Proposition 4.1 of \cite{pecastaing_conformal_lattice}). We then consider its vertical Lyapunov exponents $\chi_1,\ldots,\chi_r \in \a^*$. One of the key steps of the proof of the main result of \cite{BFH} then implies that there exists $X \in \a$ such that $\chi_1(X) = \cdots = \chi_r(X)=-1$. Considering a recurrent point $x^{\alpha} \in K^{\alpha}$ and local stable manifolds of the corresponding flow on $M^{\alpha}$, we get ``pseudo-return'' times $T_k$ for $\phi_X^t$. Translating this in terms of dynamics in $M$, we get the announced sequence $(\gamma_k)$ (see Section 6.2 of \cite{pecastaing_conformal_lattice}).
\end{proof}

\begin{definition}
Let $N$ be a Riemmanian manifold and $x \in N$. A \textit{uniformly Lyapunov regular data at $x$} is a triple $(U,(\gamma_k),(T_k))$ where $U$ is an open neighborhood of $x$, $(\gamma_k)$ a sequence in $\Diff(N)$ and $(T_k) \to \infty$ which satisfy the conclusions of Lemma \ref{lem:localisation_lyapunov_data}.
\end{definition}

The choice of the Riemannian norm is arbitrary when $N$ is compact. Here, we fix a Riemannian metric on $\X$ and $M$, and pull it back to $\tilde{\X}$ and $\tilde{M}$. We will implicitly refer to these metrics in the sequel.

For establishing Proposition \ref{prop:atlas}, we prove that a uniformly Lyapunov regular data at $x$ given by Lemma \ref{lem:localisation_lyapunov_data} gives rise, for any $\tilde{x} \in \pi^{-1}(x)$, to a maximal chart containing $\tilde{x}$. Once this is proved, applying Lemma \ref{lem:localisation_lyapunov_data} to any orbit closure $K=\bar{\Gamma.y}$, Proposition \ref{prop:atlas} follows directly.

\begin{remark}
\label{rem:ULRLinear}
If $\gamma_k.x = x$ for all $k$, then the second point of Lemma \ref{lem:localisation_lyapunov_data} means that the sequence of matrices $D_x \gamma_k \in \GL(T_xM)$ is uniformly $(T_k)$-Lyapunov regular, in the sense of Definition 6.9 of \cite{pecastaing_conformal_lattice}. 
\end{remark}

Let $(U,(\gamma_k),(T_k))$ be a uniformly Lyapunov regular data at a point $x \in M$, with $\gamma_k \in \Gamma$. Let $\tilde{x} \in \tilde{M}$ be a point over $x$. Reducing $U$ if necessary, there is a neighborhood $V$ of $\tilde{x}$ such that $\pi : V \rightarrow U$ is a diffeomorphism. For $k$ large enough, $\gamma_kU \subset U$ and there exists a unique $\tilde{\gamma_k} \in \tilde{\Gamma}$ projecting to $\gamma_k$ and such that $\tilde{\gamma}_k(\tilde{x}) \in V$. It follows that $\tilde{\gamma}_k V \subset V$ because $\pi(\tilde{\gamma}_kV) \subset U$. And since $\pi$ conjugates smoothly the action of $\tilde{\gamma}_k$ on $V$ to that of $\gamma_k$ on $U$, we get that $(V,(\tilde{\gamma_k}),(T_k))$ is a uniformly Lyapunov regular data at $\tilde{x}$.

Let $g_k = \rho(\tilde{\gamma}_k)$. If $U$ is small enough, $D$ realizes a diffeomorphism from $V$ onto its image $W \subset \X$. Then, $g_k$ preserves $W$ and has the same dynamical property as $\tilde{\gamma}_k|_V$, \textit{i.e.} $(W,(g_k),(T_k))$ is a uniformly Lyapunov regular data at $x_0 := \Db(\tilde{x}) \in \X$.

\subsection{Uniformly Lyapunov regular data on $\X$}
\label{ss:uniformly_LR_on_X}

We now consider such dynamical data on the model space $\X$. We start with some notations.

For $\X = \R P^n$, we choose $x_0=[1:0:\ldots:0]$ as an origin and note $P < G_{\X} = \PGL(n+1,\R)$  its stabilizer. We note $\a \subset \p$ the Cartan subspace of $\g_{\X}$ formed of traceless diagonal matrices. We note 
\begin{align*}
\n_- = \left \{
\begin{pmatrix}
0 & 0 \\
v & 0
\end{pmatrix}
, \ v \in \R^n
\right \} \subset \sl(n+1,\R)
\text{ and } 
\p_+ = \left \{
\begin{pmatrix}
0 & \, ^{t} \! v \\
0 & 0
\end{pmatrix}
, \ v \in \R^n
\right \} \subset \p.
\end{align*}

For $\X = \Ein^{p,q}$, we use the coordinates of $\R^{p+1,q+1}$ introduced in Section \ref{s:preliminaries}, and we also note $x_0 = [1:0:\ldots:0]$ and $P < G_{\X} = \PO(p+1,q+1)$ its stabilizer. We note $\a < \p$ the Cartan subspace of $\g_{\X}$ formed of diagonal matrices of the form
\begin{equation*}
\begin{pmatrix}
\mu_0 & & & & & \\
 & \ddots & & & & & \\
 & & \mu_p & & & & \\
 & & & 0 & & & \\
 & & & & -\mu_p & & \\
 & & & & & \ddots & \\
 & & & & & & -\mu_0
\end{pmatrix}
\text{ with }
\mu_0,\ldots,\mu_p \in \R
\text{ and the } 0 \text{ of size } q-p.
\end{equation*}
in the coordinates introduced in Section \ref{s:preliminaries}. We also note
\begin{equation*}
\n_- = 
\left \{
\begin{pmatrix}
0 & 0 & 0 \\
v & 0 & 0 \\
0 & - \, ^{t} \! v J_{p,q} & 0
\end{pmatrix}
, \ v \in \R^{p,q}
\right \} \subset \so(p+1,q+1)
\end{equation*}
and
\begin{equation*}
\p_+ = 
\left \{
\begin{pmatrix}
0 & - \, ^{t} \! v J_{p,q} & 0 \\
0 & 0 & v \\
0 & 0 & 0
\end{pmatrix}
, \ v \in \R^{p,q}
\right \} \subset \p.
\end{equation*}
In both cases, $\n_-$ and $\p$ are supplementary and $\p_+$ is the nilradical of $\p$. We note $P_+ = \exp(\p+)$, and $G_0 < P$ the section of $P/P_+$ whose Lie algebra is
\begin{align*}
\g_0 & = \left \{
\begin{pmatrix}
-\Tr(A) & \\
 & A
\end{pmatrix}
, \ A \in \mathfrak{gl}(n,\R) 
\right \} \text{ for } \X = \R P^n \\
\g_0 & = \left \{
\begin{pmatrix}
\mu_0 & & \\
 & A & \\
 & & - \mu_0
\end{pmatrix}
, \ A \in \so(p,q), \ \mu_0 \in \R\}
\right \}
 \text{ for } \X = \Ein^{p,q}.
\end{align*}
Finally, we note $W_0 \subset \X$ the image of the map $X \in \n_- \mapsto e^X.x_0$, which is a diffeomorphism onto its image. It is an affine chart domain when $\X = \R P^n$ and a Minkowski patch when $\X = \Ein^{p,q}$.

\begin{lemma}
\label{lem:gk}
Let $(g_k)$ be a sequence in $G_{\X}$ and $x \in \X$ such that:
\begin{enumerate}
\item for all sequence $x_k\rightarrow x$, we have $g_k.x_k \rightarrow x$,
\item for all non-zero $v \in T_{x}\X$, $\frac{1}{T_k} \log \|D_{x}g_k v\| \rightarrow -1$.
\end{enumerate}
Then, up to passing to a subsequence, there exists $W_{\max} \ni x$ which is an affine chart domain if $\X = \R P^n$ or a Minkowski patch if $\X = \Ein^{p,q}$, and such that for all compact subset $K \subset W_{\max}$, $g_kK \rightarrow \{x\}$ for the Hausdorff topology.
\end{lemma}

\begin{proof}

By homogeneity, we may assume $x=x_0$. We first prove that there exists $X_k \in \n_-$, with $(X_k)\rightarrow 0$, bounded sequences $(l_k),(l_k')$ in $P$, and a sequence $(a_k)$ in $A$ such that 
\begin{equation}
\label{eq:forme_de_g_k}
g_k = e^{X_k}l_ka_kl_k' \text{ for all } k.
\end{equation}
In the conformal case, this is in fact a basic case of Lemma 4.3 of \cite{frances_degenerescence}. The same approach works similarly in the projective case. We nonetheless explain how it works in this model situation. For $k$ large enough, $g_k x_0 \in W_0$ and there exists a unique $X_k \in \n_-$ such that $g_kx_0 = e^{X_k}x_0$, and $X_k \rightarrow 0$. Then, $p_k := e^{-X_k}g_k \in P$ has the same properties as $g_k$. Indeed, if we note $g_k' = e^{-X_k}$, then $g_k'$, seen as diffeomorphisms of $\X$, are bounded in topology $C^1$ since $g_k' \rightarrow \id$ in the Lie group. Thus, there is $C>0$ such that $\frac{1}{C}\| v\| \leq \|D_xg_k' v\| \leq C \|v\|$ for all $k \geq 0$ and $(x,v)$ tangent vector of $\X$. The property on the exponential growth rate of $D_{x_0}p_k$ follows directly. Also, for all $x_k \rightarrow x_0$, we have $g_kx_k \rightarrow x_0$ by assumption, and there exists $h_k \in G$, with $h_k \rightarrow \id$ such that $g_kx_k = h_kx_0$, proving that $e^{-X_k}g_kx_k = e^{-X_k}h_kx_0 \rightarrow x_0$ since $e^{-X_k}h_k \rightarrow \id$.

We decompose $p_k = p_k^{\ell} e^{Y_k}$ where $p_k^{\ell} \in G_0$ and $Y_k \in \p_+$ according to $P = G_0 \ltimes P_+$. The $KAK$ decomposition of $G_0$ gives bounded sequences $(l_k),(m_k) \in G_0$, and a sequence $a_k \in A$ such that $p_k^{\ell} = l_k a_k m_k$. Thus, $p_k = l_k a_ke^{Y_k'}m_k$, where $Y_k'=\Ad(m_k)Y_k$. Let $p_k' = a_k e^{Y_k'}$. We claim that $(p_k')$ satisfies the same hypothesis as $(g_k)$. Indeed, if $x_k \rightarrow x_0$, then writing $x_k = e^{X_k'}x_0$ for some $X_k' \in \n_-$, such that $X_k' \rightarrow 0$, we get $m_k^{-1}x_k = e^{\Ad({m_k}^{-1})X_k'}x_0$ since $m_k \in P$, and $\Ad({m_k}^{-1})X_k' \rightarrow 0$ as $\Ad({m_k}^{-1})$ is bounded. This proves that ${m_k}^{-1}x_k \rightarrow x_0$. Consequently, $p_k{m_k}^{-1}x_k \rightarrow x_0$, and finally $p_k'x_k = l_k^{-1}p_k{m_k}^{-1}x_k \rightarrow x_0$ by the same argument. The property on the exponential growth rate is also preserved because $D_{x_0}l_k$ and $D_{x_0}m_k$ are bounded sequences in $\GL(T_{x_0}\X)$ (see Remark \ref{rem:ULRLinear} and Lemma 6.10 of \cite{pecastaing_conformal_lattice}).

Using this property of $p_k'$, we prove now that $Y_k'$ is a bounded sequence of $\p_+$, which will establish (\ref{eq:forme_de_g_k}). As we said, we only have to verify this in the case $\X = \R P^n$. We note
\begin{equation*}
a_k =
\begin{pmatrix}
\lambda_0^{(k)} & & \\
& \ddots & \\
& & \lambda_{n}^{(k)}
\end{pmatrix}
\text{ and }
e^{Y_k'} =
\begin{pmatrix}
1 & v^{(k)} \\
0 & \id 
\end{pmatrix}
\end{equation*}
where $\lambda_i^{(k)}>0$ and $v^{(k)} \in \R^n$. We assume to the contrary that some component $v_i^{(k)}$ of $v^{(k)}$ is unbounded. Up to an extraction, we may assume $|v_i^{(k)}| \rightarrow \infty$. We get a contradiction with the first property of $p_k'$ by considering its action on the projective line
\begin{equation*}
p_k' [1:0:\ldots:t:\ldots:0] = [\lambda_0^{(k)}(1+v_i^{(k)}t):0:\ldots:\lambda_i^{(k)}:\ldots:0],
\end{equation*}
where $t$ stands at the $(i+1)$-th position. For $k$ large enough, we can consider $x_k := [1:0:\cdots:-\frac{1}{v_i^{(k)}}:\ldots:0]$ and we get that $p_k' x_k = [0:\ldots:1:\ldots:0]$ does not converge to $x_0$, whereas $x_k \rightarrow x_0$ since $|v_i^{(k)}| \rightarrow \infty$, a contradiction.

Finally, $e^{Y_k'} \in P$ is bounded, and if we set $l_k' = e^{Y_k'}m_k$, we get as announced
\begin{equation*}
g_k= e^{X_k}p_k = e^{X_k} l_k p_k' m_k = e^{X_k} l_k a_k l_k'
\end{equation*}
where $X_k \in \n_-$ goes to $0$, $a_k \in A$, and $l_k,l_k' \in P$ are bounded sequences.

\vspace{.2cm}

We note $\rho : P \rightarrow \GL(\g_{\X}/\p)$ the map obtained by inducing the adjoint representation of $P$ on $\g_{\X}/\p$. We remind that $\rho$ is conjugate to the isotropy representation $P \rightarrow \GL(T_{x_0}\X)$ via the identification $T_{x_0}\X \simeq \g_{\X} / \p$ given by the orbital map at $x_0$.

\begin{claim}
The sequence $\rho(a_k) \in \GL(\g_{\X}/\p)$ is $(T_k)$-uniformly Lyapunov regular (see Remark \ref{rem:ULRLinear}).
\end{claim}

By Lemma 6.10 of \cite{pecastaing_conformal_lattice}, it is the same as saying that $\rho(l_k a_k l_k') = \rho(p_k)$ is uniformly Lyapunov regular. And this was observed at the beginning of the proof, proving this claim.

\vspace{.2cm}

The action of $\rho(a_k)$ on $\g_{\X}/\p$ is the same as $\Ad(a_k)$ on $\n_-$. Writing
\begin{equation*}
\Ad(a_k)|_{\n_-} = \diag(\mu_1^{(k)},\dots,\mu_n^{(k)}),
\end{equation*}
the previous claim means $\frac{1}{T_k} \log \mu_i^{(k)} \rightarrow -1$ for all $i \in \{1,\ldots,n\}$. This implies that for any compact subset $\mathcal{K} \subset \n_-$, $\Ad(a_k)\mathcal{K} \rightarrow \{0\}$ for the Hausdorff topology, because for $k$ large enough, $\mu_i^{(k)} \leq e^{-T_k/2}$.

Up to an extraction, we may assume $l_k' \rightarrow l' \in P$. Let $W_{\max} := l'^{-1}W_0$. We prove now that for all compact subset $K \subset W_{\max}$, $g_kK \rightarrow \{x_0\}$ for the Hausdorff topology. 

Let $V \ni x_0$ be a neighborhood of $x_0$. As $X_k \rightarrow 0$, there is $k_0$ and another neighborhood $V_0 \ni x_0$ such that for all $k \geq k_0$, $e^{X_k}V_0 \subset V$. Reducing $V_0$ if necessary, we may assume $V_0 = \exp(\mathcal{V}_0).x_0$, for some $\mathcal{V}_0 \subset \g_{\X}$ neighborhood of $0$. Since $l_k$ is relatively compact in $P$, we can choose a smaller neighborhood $\mathcal{V}_1$ such that for all $k$, $\Ad(l_k)\mathcal{V}_1 \subset \mathcal{V}_0$. Hence, $(l_k \exp(\mathcal{V}_1)).x_0 = \exp(\Ad(l_k)\mathcal{V}_1).x_0 \subset V_0$. Let $V_1 = \exp(\mathcal{V}_1).x_0$.

Let $K' = l'K \subset W_0$. Let $K'' \subset W_0$ be a compact subset and $k_1$ be such that for all $k \geq k_1$, $l_k' K = (l_k'l'^{-1}) K' \subset K''$. Let $\mathcal{K}'' \subset \n_-$ be such that $K'' = \exp(\mathcal{K}'')x_0$ and let $k_2$ such that $\Ad(a_k)\mathcal{K}'' \subset \mathcal{V}_1$ for all $k \geq k_2$, so $a_kK'' \subset V_1$.

For $k \geq \max(k_0,k_1,k_2)$, we get $g_k.K = e^{X_k} l_ka_k l_k' K   \subset e^{X_k} l_ka_k K'' 
 \subset e^{X_k}l_k V_1 \subset V$, completing the proof of Lemma \ref{lem:gk}.
\end{proof}

\subsection{Maximal chat at $\tilde{x}$}

We remind that we are considering a uniformly regular Lyapunov data $(V,(\gk),(T_k))$ at a point $\tilde{x} \in \tilde{M}$ and that we note $x_0 = \Db(\tilde{x})$, $g_k = \rho(\gk)$ and $W = \Db(V)$. Since $(W,(g_k),(T_k))$ is a uniformly regular Lyapunov data at $x_0$, we consider $W_{\max} \subset \X$ the open set given by Lemma \ref{lem:gk}. Restricting $V$ if necessary, we assume $\bar{W} \subset W_{\max}$.

Consider for $k \geq 0$ the open neighborhood $V_k = \gk^{-1}V \subset \tilde{M}$ of $\tilde{x}$. By equivariance, $\Db$ is injective in restriction to $V_k$. Also, since $\{\gk V\} \rightarrow \{\tilde{x}\}$, we may assume $\gk V \subset V$ for all $k$, and then $V \subset V_k$ for all $k$.  We introduce now 
\begin{equation*}
V_{\infty} = D^{-1}(W_{\max}) \cap \bigcup_{k \geq 0} \bigcap_{l \geq k} V_l.
\end{equation*}
\begin{claim}
$V_{\infty}$ is a maximal chart containing $\tilde{x}$ and such that $D(V_{\infty}) = W_{\max}$.
\end{claim}

\begin{proof}
The injectivity of $D$ in restriction to $V_{\infty}$ is immediate as for any two points in $V_{\infty}$, there is $k \geq 0$ such that they both belong to $V_k$. 

To see that it is open, let us prove that for all $k \geq 0$, every $\tilde{y} \in \Db^{-1}(W_{\max}) \cap \bigcap_{l\geq k} V_l$ admits a neighborhood contained in $\Db^{-1}(W_{\max}) \cap \bigcap_{l \geq k'} V_l$, for some $k'$. Let $y_0 = \Db(\tilde{y})$. By definition, $y_0 \in W_{\max}$ and for all $l\geq k$, $\gl \tilde{y} \in V$. We then choose a connected open neighborhood $V_0$ of $\tilde{y}$ such that $\overline{D(V_0)} \subset W_{\max}$. By Lemma \ref{lem:gk}, there is $k'$ such that for all $l \geq k'$, $g_l\overline{D(V_0)} \subset W$. Consequently, $\Db(\gl V_0) \subset W$ for $l \geq k'$ and $\gl V_0 \cap V \neq \emptyset$ if $l \geq k$.

Since $D$ is injective on $V$, Lemma \ref{lem:assiettes_emboitees} implies that for $l \geq \max(k,k')$, we have $\gl V_0 \subset V$, \textit{i.e.} $V_0 \subset \bigcap_{l \geq \max(k,k')} V_l$, and then $V_0 \subset V_{\infty}$ proving that the latter is open.

Let us prove now that $\Db(V_{\infty}) = W_{\max}$. Let $Z \subset W_{\max}$ be a connected open subset such that $W \subset \overline{Z} \subset W_{\max}$. There exits $k_0 \geq 0$ such that $g_k.\overline{Z} \subset W$ for all $k \geq k_0$. So, $Z \subset \Db(V_k)$ for all $k \geq k_0$. Consider now $V_{k,Z} = (\Db|_{V_k})^{-1}(Z)$ which is well defined since $\Db$ is injective in restriction to $V_k$. Note that $V_{k,Z}$ is connected. We claim that $V_{k,Z} = V_{l,Z}$ for all $k,l \geq k_0$. Indeed, $\Db(\gk V_{l,Z}) =
 g_k Z \subset W$. By Lemma \ref{lem:assiettes_emboitees}, we get $\gk V_{l,Z} \subset V$ because $V \subset V_{l,Z}$ implies $V \cap \gk V_{l,Z} \neq \emptyset$, and then $V_{l,Z} \subset V_k$. Thus, $V_{l,Z} = V_{k,Z}$ since $\Db(V_{l,Z}) = Z$.

Consequently, $V_{k_0,Z} \subset \bigcap_{k \geq k_0} V_k$ and $\Db(V_{k_0,Z}) = Z$. Thus, $Z \subset \Db(V_{\infty})$, and this for all connected, relatively compact, open subset $Z \subset W_{\max}$, proving $W_{\max} \subset \Db(V_{\infty})$.

Finally, for all $k$, the covering $\pi : \tilde{M} \rightarrow M$ is injective in restriction to $V_k$ since $\pi(\gk^{-1} \tilde{y}) = \pi(\gk^{-1} \tilde{z})$ implies $\gamma_k \pi(\tilde{y}) = \gamma_k \pi(\tilde{z})$ and since $\pi|_{V}$ is injective. The same argument as for the injectivity of $\Db|_{V_{\infty}}$ then gives the injectivity of $\pi|_{V_{\infty}}$, proving the claim: $U_{\infty} = \pi(V_{\infty})$ is a maximal chart at $x$.
\end{proof}

\subsection{Conclusion}

We can conclude the proof of Proposition \ref{prop:atlas}. Let $y \in M$. Applying Lemma \ref{lem:localisation_lyapunov_data} to $\bar{\Gamma.y}$ yields a Lyapunov regular data $(U,(\gamma_k),(T_k))$ at a point $x \in \bar{\Gamma.y}$. This section has proved that there exists a maximal chart $U_{\infty}$ that contains $x$. If $\gamma \in \Gamma$ be such that $\gamma . y \in U_{\infty}$, then $\gamma^{-1} U_{\infty}$ is a maximal chart containing $y$, and the proof of Proposition \ref{prop:atlas} is complete.

\section{Injectivity of the developing map}
\label{s:injectivity}

In this section, we prove Theorem \ref{thm:atlas}. As explained in the introduction, combined with Proposition \ref{prop:atlas}, Proposition \ref{prop:projective_flat} below and Theorem 1 of \cite{pecastaing_conformal_lattice}, this will conclude the proof of Theorem \ref{thm:projective} and Theorem \ref{thm:conform}. We remind that we have fixed $\D : \tilde{M} \rightarrow \tilde{\X}$ a developing map, with holonomy $\tilde{\rho} : \Aut(\tilde{M}) \rightarrow \Aut(\tilde{\X})$. Assuming that every point of $\tilde{M}$ is contained in a maximal chart, we claim that is enough to prove that $\D$ is injective to get the conclusion. 

Indeed, if $V \subset \tilde{M}$ is a maximal chart, then $\gamma V \cap V = \emptyset$ for any non-trivial $\gamma \in \pi_1(M)$ by definition. By injectivity of $\tilde{\rho}$ and $\D$, the $\tilde{\rho}(\gamma) \D(V)$, $\gamma \in \pi_1(M)$ are pairwise disjoint. By definition, $\D(V)$ is an hemisphere of $\S^n$ in the projective case, and a Minkowski patch of $\Ein^{p,q}$ in the conformal one. Consequently $|\pi_1(M)| \leq 2$, $\tilde{M}$ is compact and $\D$ is a diffeomorphism. The conclusion follows directly.

\vspace{.2cm}

So, Theorem \ref{thm:atlas} is reduced to the proof of the injectivity of $\D$, which we establish in this section.

\subsection{Common principle}
\label{ss:injectivity:common_principle}

Let $(V_m)$ be a covering of $\tilde{M}$ by pairwise distinct maximal charts such that for all $m \geq 1$, $V_{m+1}$ intersects $\cup_{k \leq m} V_k$. We remind that the $V_i$'s are relatively compact in $\tilde{M}$ by Proposition \ref{prop:relative_compactness}. If for all $m \geq 1$, $\D(V_{m+1}) \cap (\D(V_1) \cup \cdots \cup \D(V_m))$ is connected, then, using Lemma \ref{lem:intersection_connexe}, we get by induction that $\D$ is injective in restriction to $V_1 \cup \cdots \cup V_m$ for all $m$, \textit{i.e.} that $\D$ is injective.

So, let us assume that there exists $m$ such that $\D(V_{m+1}) \cap (\D(V_1) \cup \cdots \cup \D(V_m))$ is not connected, and let us choose the smallest one. Then, by the same argument as above, we get that $\D$ is injective in restriction to $V_1 \cup \cdots \cup V_m$. Note that $m \geq 2$ by construction. 

We pick a chart $\varphi : \D(V_{m+1}) \rightarrow \R^n$ which is either an affine chart in the projective case, or a stereographic projection in the conformal case, and we note for $1 \leq i \leq m$, $U_i = \varphi(\D(V_i) \cap \D(V_{m+1}))$.

\begin{claim}
The $U_i$'s are pairwise distinct.
\end{claim} 

Indeed, if $U_i = U_j$, then $\D(V_i) = \D(V_j)$ because an hemisphere (resp. a Minkowski patch) is determined by its intersection with a given hemisphere (resp. Minkowski patch). By injectivity of $\D$ on $V_1 \cup \ldots \cup V_m$, this implies $V_i=V_j$ and then $i=j$ by choice of $(V_i)$.

\vspace{.2cm}

We then classify the configurations in which a family of such open subsets of $\R^n$ can have non-connected union. Finally we prove that in such configurations, if one of the $V_i$'s is thickened (see Section \ref{ss:thickenings}), then Lemma \ref{lem:intersection_connexe} applies and yields an open set $V \subset \tilde{M}$ in restriction to which $\D$ is injective and such that $\D(V) = \tilde{\X}$, proving that $V = \tilde{M}$ by Lemma \ref{lem:assiettes_emboitees}, and completing the proof of the injectivity of $\D$ in this \textit{a priori} problematic situation.

\subsection{Projective case}
\label{ss:injectivity:projective_case}

For $\X = \R P^n$, $\D$ sends maximal charts onto hemispheres of $\S^n$. We will use the basic facts recalled below. Let $\iota : \S^n \rightarrow \S^n$ be the antipodal map. 

\subsubsection{Some conventions and facts on hemispheres}

We embed $\S^n \subset \R^{n+1}$ in the standard way. A hemisphere $H \subset \S^n$ is the data of a half-line $\R_{>0} \ell$, for $\ell \in (\R^{n+1})^*$ such that $H = \S^n \cap \{\ell > 0\}$.

Given two hemispheres $H_0$ and $H_1$, and an affine chart $\varphi : H_0 \rightarrow \R^n$, if $H_0$ and $H_1$ are not equal or antipodal, then $\varphi(H_0 \cap H_1)$ is an affine half-space of $\R^n$. This gives a bijection between the set of hemisphere minus $\{H_0,\iota(H_0)\}$ and the set of affine half-spaces of $\R^n$.

We will use later the following facts which can be easily observed in coordinates.

\begin{fact}
\label{fact:3hemispheres}
Let $H_0,H_1,H_2$ be three hemispheres such that $H_1 \neq \iota(H_2)$. If $H_0 \cap H_1$ and $H_0 \cap H_2$ are disjoint, then $\iota(H_0) \subset H_1 \cup H_2$.
\end{fact}

\begin{fact}
\label{fact:equator}
Let $H_0,H_1,H_2$ be three hemispheres and let $\varphi : H_0 \rightarrow \R^n$ be an affine chart. Assume that $\varphi(H_0 \cap H_1)$ and $\varphi(H_0 \cap H_2)$ are parallel and in the same direction, that is there exists $\phi \in (\R^n)^*$, $\alpha_1,\alpha_2 \in \R$ such that $\varphi(H_0 \cap H_i) = \{\phi > \alpha_i\}$ for $i =1,2$. Then, $H_1 \cap \partial H_0 = H_2 \cap \partial H_0$.
\end{fact}

Let $V \subset \tilde{M}$ be a relatively compact maximal chart, and let $H = \D(V)$. We have seen in Section \ref{ss:thickenings} that $\D$ is still injective on small enough neighborhoods of $\bar{V}$. In particular, if $\epsilon > 0$ is small enough, there is a neighborhood $V^{\epsilon}$ of $\bar{V}$ on which $\D$ is injective and such that $\D(V^{\epsilon}) = H^{\epsilon} := \S^n \cap \{\ell > -\epsilon\}$ where $\ell$ is such that $H = \S^n \cap \{\ell > 0\}$.

Affine charts are not well adapted to these thickenings $H^{\epsilon}$, it is more relevant to use stereographic projections even though no conformal structure is involved. 

Notably, if $x \notin H^{\epsilon}$, and if $s : \S^n \setminus \{x\} \rightarrow \R^n$ is a stereographic projection, then $s(H^{\epsilon})$ is a ball. The following fact is then clear.

\begin{fact}
\label{fact:thickned_hemisphere1}
Let $H_1,H_2$ be two hemispheres. If $\epsilon > 0$ is small enough, then $H_1^{\epsilon} \cap H_2$ is connected.
\end{fact}

\begin{proof}
If $H_2 = \iota(H_1)$, then we pick $x \notin H_1^{\epsilon}$ and fix a stereographic projection $s:  \S^n \setminus \{x\} \rightarrow \R^n$. Then, $s(H_2 \setminus \{x\})$ is the complement of a closed ball $\bar{B}_1$ and $s(H_1^{\epsilon})$ is another open ball $B_2$, that contains $\bar{B}_1$. So, $H_1^{\epsilon} \cap H_2$ is diffeomorphic to $B_2 \setminus \bar{B}_1$, which is connected.

If $H_2 \neq \iota(H_1)$, then for $\epsilon >0$ small enough, we can choose $x \notin H_1^{\epsilon} \cup H_2$. A stereographic projection defined on $\S^n \setminus \{x\}$ then sends $H_1^{\epsilon} \cap H_2$ onto the intersection of two balls in $\R^n$, which is connected.
\end{proof}

Finally, we will make use of the following.

\begin{fact}
\label{fact:thickened_hemisphere2}
Let $H_0,H_1,H_2$ be three hemispheres, with $H_1$ and $H_2$ not antipodal. Assume that $H_0 \cap (H_1 \cup H_2)$ is not connected. Then, for small enough $\epsilon$, $H_0^{\epsilon} \cap (H_1 \cup H_2)$ is connected. See Figure \ref{fig:3hemispheres}.
\end{fact}

\begin{proof}
By Fact \ref{fact:thickned_hemisphere1}, it is enough to prove that $H_0^{\epsilon} \cap H_1 \cap H_2$ is non-empty for small enough $\epsilon$. By assumption, $H_0 \cap H_1$ and $H_0 \cap H_2$ are disjoint. Let $\ell_0,\ell_1,\ell_2$ be linear forms defining $H_0,H_1,H_2$ respectively. Our assumption means that $\ell_1$ and $\ell_2$ are non-collinear and $-\ell_0 \in \Conv(\ell_1,\ell_2)$ (the open convex hull).

In particular, there is a point $x \in \cap_i \partial H_i$. Since $x \in H_0^{\epsilon}$, it is enough to observe that $H_1$ and $H_2$ intersect arbitrarily close to $x$. To see it, we pick $s : \S^n \setminus \{-x\} \rightarrow \R^n$ a stereographic projection such that $s(x) = 0$. Then, $H_1$ and $H_2$ are sent to half-spaces delimited by two distinct linear hyperplanes. It is then immediate that they intersect arbitrarily close to $0$.
\end{proof}

\subsubsection{Configurations for which the induction fails}

For all $i \leq m+1$, we note $\D(V_i)=H_i \subset \S^n$. As announced above, we consider the smallest integer $m$ such that $H_{m+1} \cap (H_1 \cup \ldots H_m)$ is not connected. Let $\varphi : H_{m+1} \rightarrow \R^n$ be an affine chart. For all $1 \leq i \leq m$, $U_i = \varphi(H_i \cap H_{m+1}) \subset \R^n$ is either empty, a half-space, or $\R^n$. 

We remind that $U_1, \ldots, U_m$ are pairwise distinct, in particular at most one of them is empty. Thus, there is $l \in \{m-1,m\}$, with $l \geq 2$, an injective map $\sigma : \{1,\ldots,l\} \rightarrow \{1,\ldots,m\}$, $\phi \in (\R^n)^*$, $k_0 \in \{1,\ldots,l-1\}$ and $\alpha_1 < \cdots < \alpha_{k_0} \leq \alpha_{k_0+1} < \cdots < \alpha_l$ such that $U_{\sigma(k)} = \{\phi < \alpha_k\}$ for $1 \leq k \leq k_0$, $U_{\sigma(k)} = \{\phi > \alpha_k\}$ for $k_0+1 \leq k \leq l$, and if $l = m-1$ and $i$ is the unique element not in the range of $\sigma$, $U_i = \emptyset$. We note $i_0 = \sigma(k_0)$ and $j_0 = \sigma(k_0+1)$.

\subsubsection{Case $\alpha_{k_0} = \alpha_{k_0+1}$}

In this situation, $H_{i_0}$ and $H_{j_0}$ are antipodal. Since $H_1 \cup \cdots \cup H_m$ is connected, it contains a point $x_0 \in \partial H_{i_0} = \partial H_{j_0}$. By injectivity of $\D$ in restriction to $V_1 \cup \cdots \cup V_m$, if $x \in V_1 \cup \cdots \cup V_m$ is the preimage of $x_0$, then $x \in \partial V_{i_0} \cap \partial V_{j_0}$, showing that $\partial V_{i_0} \cap \partial V_{j_0} \neq \emptyset$.

Now, Lemma \ref{lem:thickenings} gives a neighborhood $V_{i_0}^{\epsilon} \supset \bar{V_{i_0}}$ in restriction to which $\D$ is injective and such that $\D(V_{i_0}^{\epsilon}) = H_{i_0}^{\epsilon}$ for some $\epsilon > 0$. Then $V_{i_0}^{\epsilon} \cap V_{j_0} \neq \emptyset$ and $H_{i_0}^{\epsilon} \cap H_{j_0}$ is connected by Fact \ref{fact:thickned_hemisphere1}. By Lemma \ref{lem:intersection_connexe}, we get that $\D$ is injective in restriction to $V_{i_0}^{\epsilon} \cup V_{j_0}$, and $\D(V_{i_0}^{\epsilon} \cup V_{j_0}) = \S^n$.

Thus, we get that $V_{i_0}^{\epsilon} \cup V_{j_0} = \tilde{M}$ and that $\D$ is a diffeomorphism onto $\S^n$.

\subsubsection{Case $\alpha_{k_0} < \alpha_{k_0+1}$}

In this situation, we claim that $H_1 \cup \cdots \cup H_m = H_{i_0} \cup H_{j_0}$. Indeed, considering the partition $\S^n = H_{m+1} \cup \partial H_{m+1} \cup \iota(H_{m+1})$, we see first that $H_{m+1} \cap (H_1 \cup \cdots \cup H_m)$ and $H_{m+1} \cap (H_{i_0} \cup H_{j_0})$ coincide by assumption and choice of $k_0$. Then, by Fact \ref{fact:3hemispheres}, we get $\iota(H_{m+1}) \subset H_{i_0} \cup H_{j_0}$, proving in particular that $\iota(H_{m+1}) \cap (H_1 \cup \cdots \cup H_m)$ and $\iota(H_{m+1}) \cap (H_{i_0} \cup H_{j_0})$ also coincide. Finally, by Fact \ref{fact:equator}, we have $H_{\sigma(1)} \cap \partial H_{m+1} = \cdots = H_{\sigma(k_0)} \cap \partial H_{m+1}$ and $H_{\sigma(k_0+1)} \cap \partial H_{m+1} = \cdots = H_{\sigma(l)} \cap \partial H_{m+1}$, proving that $\partial H_{m+1} \cap (H_{i_0} \cup H_{j_0}) = \partial H_{m+1} \cap (H_1 \cup \cdots \cup H_m)$.

By injectivity of $\D$ in restriction to $V_1 \cup \cdots \cup V_m$, we deduce $V_{i_0} \cup V_{j_0} = V_1 \cup \cdots \cup V_m$. Therefore, $V_{i_0} \cap V_{j_0} \neq \emptyset$ and $V_{m+1} \cap (V_{i_0} \cup V_{j_0}) \neq \emptyset$ and $H_{m+1} \cap (H_{i_0} \cup H_{j_0})$ is not connected. 

Therefore, Fact \ref{fact:thickened_hemisphere2} and Lemma \ref{lem:thickenings} imply that $\bar{V_{m+1}}$ admits a neighborhood $V_{m+1}^{\epsilon}$ in restriction to which $\D$ is injective and such that $\D(V_{m+1}^{\epsilon}) \cap (\D(V_{i_0} \cup \D(V_{j_0}))$ is connected. Thus, we can apply Lemma \ref{lem:intersection_connexe} to obtain first that $\D$ is injective in restriction to $V_{i_0} \cup V_{j_0}$, and then in restriction to $V_{m+1}^{\epsilon} \cup V_{i_0} \cup V_{j_0}$. The image of this open subset of $\tilde{M}$ is $\S^n$, proving that $\tilde{M} = V_{m+1}^{\epsilon} \cup V_{i_0} \cup V_{j_0}$ and that $\D : \tilde{M} \rightarrow \S^n$ is a diffeomorphism.

\begin{figure}
\begin{tikzpicture}[line cap=round,line join=round,>=triangle 45,x=1cm,y=1cm,scale=.4]

\clip(-15,-10) rectangle (14,10);

\draw [line width=1pt] (-7.78,-2.76) circle (5.35372767331324cm); 

\node[scale=1] at (-7.7,3.2) {$H_{m+1}$};

\draw [shift={(-7.78,-2.76)},line width=1pt]  plot[domain=-3.141592653589793:0,variable=\t]({0.9999721553176327*5.354550076790271*cos(\t r)+-0.007462478771026879*1.9540211495656878*sin(\t r)},{0.007462478771026879*5.354550076790271*cos(\t r)+0.9999721553176327*1.9540211495656878*sin(\t r)}); 

\draw [shift={(-7.78,-2.76)},line width=1pt,dashed]  plot[domain=0:3.141592653589793,variable=\t]({0.9999721553176327*5.354550076790271*cos(\t r)+-0.007462478771026879*1.9540211495656878*sin(\t r)},{0.007462478771026879*5.354550076790271*cos(\t r)+0.9999721553176327*1.9540211495656878*sin(\t r)}); 

\node[scale=1] at (-13.7,-4.5) {$H_{i_0}$};


\draw [shift={(-7.78,-2.76)},line width=1pt]  plot[domain=-3.141592653589793:0,variable=\t]({0.8644337869842375*5.3499149024909105*cos(\t r)+-0.502746683649022*1.4797465578484086*sin(\t r)},{0.502746683649022*5.3499149024909105*cos(\t r)+0.8644337869842375*1.4797465578484086*sin(\t r)});

\draw [shift={(-7.78,-2.76)},line width=1pt,dashed]  plot[domain=0:3.141592653589793,variable=\t]({0.8644337869842375*5.3499149024909105*cos(\t r)+-0.502746683649022*1.4797465578484086*sin(\t r)},{0.502746683649022*5.3499149024909105*cos(\t r)+0.8644337869842375*1.4797465578484086*sin(\t r)}); 

\draw [shift={(-7.78,-2.76)},line width=1pt]  plot[domain=-3.141592653589793:0,variable=\t]({0.8070088429497047*5.348325122754983*cos(\t r)+0.590539352965557*1.8915022288039713*sin(\t r)},{-0.590539352965557*5.348325122754983*cos(\t r)+0.8070088429497047*1.8915022288039713*sin(\t r)});

\draw [shift={(-7.78,-2.76)},line width=1pt,dashed]  plot[domain=0:3.141592653589793,variable=\t]({0.8070088429497047*5.348325122754983*cos(\t r)+0.590539352965557*1.8915022288039713*sin(\t r)},{-0.590539352965557*5.348325122754983*cos(\t r)+0.8070088429497047*1.8915022288039713*sin(\t r)});

\node[scale=1] at (-1.7,-5) {$H_{j_0}$};

\draw [line width=1pt] (8,-2.76) circle (5.35372767331324cm);
\draw [shift={(8,-2.76)},line width=.3pt]  plot[domain=-3.141592653589793:0,variable=\t]({0.9999721553176327*5.354550076790271*cos(\t r)+-0.007462478771026879*1.9540211495656878*sin(\t r)},{0.007462478771026879*5.354550076790271*cos(\t r)+0.9999721553176327*1.9540211495656878*sin(\t r)});
\draw [shift={(8,-2.76)},line width=.3pt,dashed]  plot[domain=0:3.141592653589793,variable=\t]({0.9999721553176327*5.354550076790271*cos(\t r)+-0.007462478771026879*1.9540211495656878*sin(\t r)},{0.007462478771026879*5.354550076790271*cos(\t r)+0.9999721553176327*1.9540211495656878*sin(\t r)});
\draw [shift={(8,-2.76)},line width=1pt]  plot[domain=-3.141592653589793:0,variable=\t]({0.8644337869842375*5.3499149024909105*cos(\t r)+-0.502746683649022*1.4797465578484086*sin(\t r)},{0.502746683649022*5.3499149024909105*cos(\t r)+0.8644337869842375*1.4797465578484086*sin(\t r)});
\draw [shift={(8,-2.76)},line width=1pt,dashed]  plot[domain=0:3.141592653589793,variable=\t]({0.8644337869842375*5.3499149024909105*cos(\t r)+-0.502746683649022*1.4797465578484086*sin(\t r)},{0.502746683649022*5.3499149024909105*cos(\t r)+0.8644337869842375*1.4797465578484086*sin(\t r)});
\draw [shift={(8,-2.76)},line width=1pt]  plot[domain=-3.141592653589793:0,variable=\t]({0.8070088429497047*5.348325122754983*cos(\t r)+0.590539352965557*1.8915022288039713*sin(\t r)},{-0.590539352965557*5.348325122754983*cos(\t r)+0.8070088429497047*1.8915022288039713*sin(\t r)});
\draw [shift={(8,-2.76)},line width=1pt,dashed]  plot[domain=0:3.141592653589793,variable=\t]({0.8070088429497047*5.348325122754983*cos(\t r)+0.590539352965557*1.8915022288039713*sin(\t r)},{-0.590539352965557*5.348325122754983*cos(\t r)+0.8070088429497047*1.8915022288039713*sin(\t r)});
\draw [shift={(8.01,-3.3460197833054655)},line width=1pt,dashed]  plot[domain=-0.0578405920835543:3.1788888063164764,variable=\t]({0.9999721553176328*5.316620430281443*cos(\t r)+-0.007462478771026865*1.9764514282475778*sin(\t r)},{0.007462478771026865*5.316620430281443*cos(\t r)+0.9999721553176328*1.9764514282475778*sin(\t r)});
\draw [shift={(8.01,-3.3460197833054655)},line width=1pt]  plot[domain=3.1788888063164764:6.225344715096032,variable=\t]({0.9999721553176328*5.316620430281443*cos(\t r)+-0.007462478771026865*1.9764514282475778*sin(\t r)},{0.007462478771026865*5.316620430281443*cos(\t r)+0.9999721553176328*1.9764514282475778*sin(\t r)});

\node[scale=.8] at (1.6,-2.8) {$\ell=0$};
\node[scale=.8] at (1.4,-3.6) {$\ell = -\epsilon$};
\node[scale=1] at (8.2,3.2) {$H_{m+1}^{\epsilon}$};
\end{tikzpicture}

\caption{Configuration for $\alpha_{k_0} < \alpha_{k_0+1}$}
\label{fig:3hemispheres}
\end{figure}
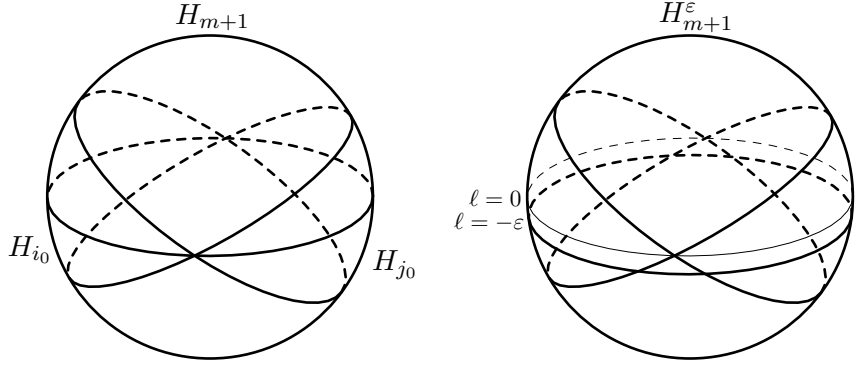

\subsection{Conformal case} 
\label{ss:injectivity:conformal_case}

For $\X = \Ein^{p,q}$, $\D$ sends maximal charts to Minkowski patches of $\tilde{\Ein}^{p,q}$. For $1 \leq i \leq m+1$, we note $M_i = \D(V_i)$. We remind that $m$ is assumed to be the smallest integer such that $M_{m+1} \cap (M_1 \cup \cdots \cup M_m)$ is not connected. We note $\iota : \tilde{\Ein}^{p,q} = \S^p \times \S^q \rightarrow \tilde{\Ein}^{p,q}$ the product of the antipodal maps.

\subsubsection{Facts about Minkowski patches}

\begin{fact}
\label{fact:iota_m0}
Let $M_0,M_1,M_2$ be three Minkowski patches such that $M_0 \cap M_1$ and $M_0 \cap M_2$ are degenerate half-spaces of $M_0$, and $M_1 \neq \iota(M_2)$. If $M_0 \cap M_1$ and $M_0 \cap M_2$ are disjoint, then $\iota(M_0) \subset M_1 \cup M_2$.
\end{fact}

\begin{proof}
Let $s : M_0 \rightarrow \R^{p,q}$ be a stereographic projection. By assumption, there is $v \in \R^{p,q}$ isotropic and $\alpha,\beta \in \R$ such that $s(M_0 \cap M_1) = H_{v,\alpha}$ and $s(M_0 \cap M_2) = H_{-v,-\beta}$, and $\alpha > \beta$ (cf. Definition \ref{def:intersection_type}). Let $s' := s \circ \iota : \iota(M_0) \rightarrow \R^{p,q}$. By Remark \ref{rem:trace_sur_lantipodal}, we get that $s'(\iota(M_0) \cap M_1) = H_{-v,-\alpha}$ and $s'(\iota(M_0) \cap M_2) = H_{v,\beta}$, showing $\iota(M_0) \cap (M_1 \cup M_2) = \iota(M_0)$, because $H_{-v,-\alpha} \cup H_{v,\beta} = \R^{p,q}$.
\end{proof}

\begin{lemma}
\label{lem:trace_sur_le_bord}
Let $M_0,M_1,M_2$ be three Minkowski patches and $s : M_0 \rightarrow \R^{p,q}$ a stereographic projection. Assume that $s(M_0 \cap M_1) = H_{v,\alpha}$ and $s(M_0 \cap M_2) = H_{v,\beta}$ for $v \in \R^{p,q}$ isotropic, and $\alpha,\beta \in \R$. Then, $\partial M_0 \cap M_1 = \partial M_0 \cap M_2$.
\end{lemma}

\begin{remark}
The condition on the intersection of the Minkowski patches is independent of the choice of $s$, in the sense that if $s'$ is another stereographic projection, then we have $v'$ isotropic and $\alpha',\beta' \in \R$ such that $s'(M_0 \cap M_1) = H_{v',\alpha'}$ and $s'(M_0 \cap M_2) = H_{v',\beta'}$.
\end{remark}

\begin{proof}
Let $M_x = \pein(M_0)$, $M_y = \pein(M_1)$ and $M_z = \pein(M_2)$ be their projections in $\Ein^{p,q}$, with vertices $x,y,z \in \Ein^{p,q}$. By hypothesis, $x,y,z$ lie on a same light-like geodesic $\Delta \subset \Ein^{p,q}$. Moreover, $y \neq x$ and $z \neq x$, because $M_1$ and $M_2$ cannot be equal to $M_0$ or antipodal to $M_0$ by assumption.

As $\PO(p+1,q+1)$ acts transitively on the set of pointed light-like projective lines of $\Ein^{p,q}$, we may assume $x = [1:0:\cdots:0]$ and $\Delta = \{[s_0:s_1:0:\cdots:0],\ (s_0,s_1) \neq (0,0)\}$ in the coordinates introduced in Section \ref{s:preliminaries}. Applying $\iota$ if necessary, we may also assume 
\begin{equation*}
M_0 = \left \{s^{-1}(u) := \frac{1}{\|(-\frac{q(u)}{2},u,1)\|} \left ( -\frac{q(u)}{2},u,1\right ), \ u \in \R^{p,q} \right \}
\end{equation*}
where $\|.\|$ denotes the usual Euclidean norm on $\R^n$ and $q$ the quadratic form on $\R^{p,q}$ induced by our choice of coordinates. Note that this means that $(-1,0,\ldots,0)$ is the space-like vertex of $M_0$ and $(1,0,\ldots,0)$ its time-like vertex.

Now, there is $t \in \R$ such that $y = [t:1:0:\cdots:0]$. Noting $v_y = (t,1,0,\ldots,0) \in \R^{p+1,q+1}$, $M_1$ is one of the two connected components of $\tilde{\Ein^{p,q}} \setminus v_y^{\perp}$, where the orthogonal is taken relatively to the inner product of $\R^{p+1,q+1}$. That is:
\begin{align*}
\text{either } & M_1 = \{(s_0,\ldots,s_{n+1}) \in \tilde{\Ein}^{p,q} \ : \ s_n + ts_{n+1} > 0\} \\
\text{ or } & M_1 = \{(s_0,\ldots,s_{n+1}) \in \tilde{\Ein}^{p,q} \ : \ s_n + ts_{n+1} < 0\}.
\end{align*}
In each case, we get $s(M_0 \cap M_1) = \{u \in \R^{p,q} \ : \ u_n > -t\} = H_{e_1,-t}$ or $s(M_0 \cap M_1) = \{u \in \R^{p,q} \ : \ u_n < -t\} = H_{-e_1,t}$. Let us say that we are in the first case.

We consider now $C_0 = \partial M_0 = \tilde{\Ein}^{p,q} \cap e_0^{\perp}$. We get
\begin{align*}
C_0 = & \left  \{ \frac{1}{\|(1,x,0)\|}(1,x,0), \ x \in C^{p,q}\right \} \cup \left  \{ \frac{1}{\|(1,x,0)\|}(-1,x,0), \ x \in C^{p,q} \right \} \\
 & \cup \left  \{ (0,x,0), \ x \in \tilde{\Ein}^{p-1,q-1} \right\},
\end{align*}
where $C^{p,q}$ denotes $\{x \in \R^{p,q} \ : \ q(x)=0\}$. So, $C_0 \cap M_1$ is simply the same union, with the additional requirement that $x_n > 0$, where $x_n$ is the last coordinate of $x$. Thus, the parameter $t$ defining the position of $y$ on $\Delta$ does not appear any longer.

This finishes the proof. Indeed, there is $t' \in \R$ such that $z = [t':1:0:\cdots:0]$. Necessarily, we will have
\begin{equation*}
M_2 = \{(s_0,\ldots,s_{n+1}) \in \tilde{\Ein}^{p,q} \ : \ s_n + t's_{n+1} > 0\}
\end{equation*}
because $s(M_0 \cap M_1)$ and $s(M_0 \cap M_2)$ are assumed to be ``oriented'' by the same isotropic vector.  
Consequently, $M_2 \cap C_0 = M_1 \cap C_0$ as announced.
\end{proof}

Similarly to the projective case, we will use the fact that $\D$ is still injective on some neighborhood of the closure of the $V_i$'s. We will consider neighborhoods of closures of maximal charts which are developed to the following type of neighborhoods of closures of Minkowski patches.

\begin{definition}
Let $M_0 \subset \tilde{\Ein}^{p,q}$ be a Minkowski patch and $s : M_0 \rightarrow \R^{p,q}$ a stereographic projection. For all $\epsilon > 0$, we define $M_0^{s,\epsilon} = \tilde{\Ein}^{p,q} \setminus (s \circ \iota)^{-1}(B(0,\frac{1}{\epsilon}))$, where $B(0,R) = \{v \in \R^{p,q} \ : \ v_1^2 + \cdots + v_n^2 \leq R^2\}$ for $R > 0$.
\end{definition}

We observe that for any $s$, the $M_0^{s,\epsilon}$, $\epsilon > 0$, form a fundamental system of neighborhoods of $\bar{M_0}$. 

\begin{lemma}
For any open neighborhood $\mathcal{V} \supset \bar{M_0}$ and any stereographic projection $s : M_0 \rightarrow\R^{p,q}$, there exists $\epsilon > 0$ such that $M_0^{s,\epsilon} \subset \mathcal{V}$.
\end{lemma}

Thus, if $V \subset \tilde{M}$ is a relatively compact maximal chart and $s : \D(V) \rightarrow \R^{p,q}$ a stereographic projection, for small enough $\epsilon > 0$, there exists a neighborhood $V^{s,\epsilon}$ of $\bar{V}$ on which $\D$ is still injective and such that $\D(V^{s,\epsilon}) = \D(V)^{s,\epsilon}$.

\begin{lemma}
\label{lem:intersection_thickening_connected}
Let $M_0$, $M_1$ be two Minkowski patches in $\tilde{\Ein}^{p,q}$. Then, for all $\epsilon >0$ and stereographic projection $s : M_0 \rightarrow \R^{p,q}$, $M_0^{s,\epsilon} \cap M_1$ is a non-empty, connected open set.
\end{lemma}

\begin{proof}
The lemma is clear if $M_1 = M_0$ or $M_1 = \iota(M_0)$. We note $s': = s \circ \iota$. According to the partition $\tilde{\Ein}^{p,q} = \bar{M_0} \cup \iota(M_0)$, because $M_0^{s,\epsilon}$ is a neighborhood of the closure of $\bar{M_0}$, we have $M_0^{s,\epsilon} \cap M_1 = (M_1 \cap \bar{M_0}) \cup ((M_1 \cap \iota(M_0)) \setminus s'^{-1}(B(0,\frac{1}{\epsilon})))$.

Because $M_1 = (M_1 \cap \bar{M_0}) \cup (M_1 \cap \iota(M_0))$ is connected, it is enough to observe that $(M_1 \cap \iota(M_0)) \setminus s'^{-1}(B(0,\frac{1}{\epsilon}))$ is connected. In the stereographic projection $s' : \iota(M_0) \rightarrow \R^{p,q}$, this open set is sent to 
\begin{itemize}
\item either $(v+U_S) \setminus B(0,\frac{1}{\epsilon})$ for some $v \in \R^{p,q}$,
\item or $(v+U_T) \setminus B(0,\frac{1}{\epsilon})$ for some $v \in \R^{p,q}$,
\item or $H_{v,\alpha} \setminus B(0,\frac{1}{\epsilon})$ for some isotropic $v \in \R^{p,q}$ and $\alpha \in \R$.
\end{itemize}
All of them are always connected, proving the lemma.
\end{proof}

\subsubsection{Configurations for which the induction fails}

Let $s : M_{m+1} \rightarrow \R^{p,q}$ be a stereographic projection and $U_i = s(M_i \cap M_{m+1})$ for $1 \leq i \leq m$, so that $U_i$ is either empty or of intersection type (Definition \ref{def:intersection_type}).

\begin{lemma}
\label{lem:configuration_conforme}
Let $W_1,\ldots,W_l \subset \R^{p,q}$ be a finite family of pairwise distinct, non-empty open sets of intersection type. Then, $W_1 \cup \cdots \cup W_l$ is not connected if and only if 
\begin{enumerate}
\item either there exist $v \in \mathcal{C} \setminus \{0\}$, $\alpha_1,\ldots,\alpha_l \in \R$ and $1 \leq k_0 \leq l-1$ such that up to permutation, $\alpha_1 > \cdots > \alpha_{k_0} \geq \alpha_{k_0+1} > \cdots > \alpha_l$ and $W_i = H_{v,\alpha_i}$ for all $i \leq k_0$ and $W_i = H_{-v,-\alpha_i}$ for all $i > k_0$.
\item or $l=2$, and up to permutation, $U_1 = v + U_S$ and $U_2 = v + U_T$ for $v \in \R^{p,q}$.
\end{enumerate}
\end{lemma}

\begin{proof}
We use repeatedly Fact \ref{fact:disjoint_open_sets}. Let us assume that $U := W_1 \cup \cdots \cup W_l$ is not connected.

\textbf{Case 1}: There exists $i$ such that $W_i = H_{v_i,\alpha_i}$ for $v_i \in \R^{p,q}$ isotropic and $\alpha_i \in \R$.

In this situation, necessarily for all $j$, $W_j$ is also of the form $W_j = H_{v_j,\alpha_j}$. Indeed, let us assume for instance that, to the contrary, there exists $j$ such that $W_j = v_j + U_S$. Then, for all $1 \leq k \leq l$, $W_k$ intersects $W_i \cup W_j$. The latter is connected because $W_i \cap W_j \neq \emptyset$. Thus, given any $1 \leq k \leq n$, $W_k \cup W_i \cup W_j$ is connected, proving that $U$ is connected, a contradiction.

Moreover, by similar a argument, all the vectors $v_j$ must lie on a same isotropic line. If we rescale them, we get that up to a permutation of $\{1,\ldots,l\}$, there is $k_0 \in \{1,\ldots,l-1\}$ and $v \in \R^{p,q}$ isotropic such that for all $k \leq k_0$, $W_k = H_{v,\alpha_k}$ and for all $k \geq k_0+1$, $W_k = H_{-v,-\alpha_k}$ and with $\alpha_1 > \cdots > \alpha_{k_0} \geq \alpha_{k_0+1} > \cdots > \alpha_{l}$.

\textbf{Case 2}: For all $i$, $W_i$ is of the form $v_i + U_S$ or $v_i +U_T$.

\textbf{Case 2.a}: All $W_i$'s are of the same type. Then, they intersect pairwise and $U$ is connected, a contradiction.

\textbf{Case 2.b}: There exist $i,j$ such that $W_i = v_i+U_S$ and $W_j = v_j+U_T$. If we had $v_i \neq v_j$, then $W_i \cup W_j$ would be connected, and since any other $W_k$ would intersect it, we would get as before that $U$ is connected. So, $v_i = v_j$, and moreover, if there existed a third open subset $W_k$ distinct from $W_i$ and $W_j$, then $W_k$ would intersect $W_i$ and $W_j$, implying that $W_i \cup W_j \cup W_k$ is connected and dense in $\R^{p,q}$. In particular, it would intersect any other $W_{k'}$, contradicting the non-connectedness of $U$.

Finally, in Case 2, we must have $l=2$ and $W_1 = v+U_S$ and $W_2 = v+U_T$ as announced.
\end{proof}

\subsubsection{Case of a family of half-spaces}
\label{sss:family_halfspaces}

We assume here that if we remove the eventual $U_i$ which is empty, the remaining ones are in the first configuration of Lemma \ref{lem:configuration_conforme}. We then have $l \in \{ m-1 , m\}$, $W_1,\ldots,W_l \subset \R^n$, $1 \leq k_0 \leq l-1$, $\alpha_1 > \cdots > \alpha_{k_0} \geq \alpha_{k_0+1} > \cdots > \alpha_l$ such that $W_k = H_{v,\alpha_k}$ for $k \leq k_0$ and $W_k = H_{-v,-\alpha_k}$ for $k > k_0$, and an injective map $\sigma : \{1,\ldots,l\} \rightarrow \{1,\ldots,m\}$ such that $W_k = U_{\sigma(k)}$ for all $k$, and if $l=m-1$ and $i$ is the unique element not in the range of $\sigma$, $U_i = \emptyset$. Let $i_0,j_0$ be such that $U_{i_0} = W_{k_0}$ and $U_{j_0} = W_{k_0+1}$. 

\vspace{.2cm}

\begin{flushleft}
\underline{Case $\alpha_{k_0} = \alpha_{k_0+1}$.}
\end{flushleft}

In this situation $M_{j_0} = \iota (M_{i_0})$, and we the same arguments as in the projective case apply verbatim to get $\partial V_{i_0} \cap \partial V_{j_0} \neq \emptyset$. Consequently, if $s_0 : M_{i_0} \rightarrow \R^{p,q}$ is any stereographic projection and if $\epsilon >0$ is small enough, such that there exists a neighborhood $V_{i_0}^{s_0,\epsilon}$ of $\bar{V}_{i_0}$ in restriction to which $\D$ is injective and such that $\D(V_{i_0}^{s_0,\epsilon}) = M_{i_0}^{s_0,\epsilon}$, then $V_{i_0}^{s_0,\epsilon} \cap V_{j_0} \neq \emptyset$. The intersection $M_{i_0}^{s_0,\epsilon} \cap M_{j_0}$ is homeomorphic to the complement of a ball in $\R^{p,q}$, thus connected. We conclude by Lemma \ref{lem:intersection_connexe} that $\D$ is injective on $V_{i_0}^{s_0,\epsilon} \cup V_{j_0}$, and the image of the latter is $\tilde{\Ein}^{p,q}$.

\vspace{.2cm}

\begin{flushleft}
\underline{Case $\alpha_{k_0} < \alpha_{k_0+1}$.}
\end{flushleft}

\begin{claim}
In this situation, $M_1 \cup \ldots \cup M_m = M_{i_0} \cup M_{j_0}$.
\end{claim}

\begin{proof}
We prove the non-obvious inclusion by observing that for all $i$, the traces of $M_i$ on the partition $\tilde{\Ein}^{p,q} = M_{m+1} \cup \partial M_{m+1} \cup \iota(M_{m+1})$ are included in $M_{i_0} \cup M_{j_0}$. By Fact \ref{fact:iota_m0} and by the choice of $i_0,j_0$, we have $\iota(M_{m+1}) \subset M_{i_0} \cup M_{j_0}$, and immediately, $M_i \cap \iota(M_{m+1}) \subset M_{i_0} \cup M_{j_0}$. Applying $s$, it is immediate by construction that $M_i \cap M_{m+1} \subset M_{i_0} \cup M_{j_0}$. Finally,
\begin{enumerate}
\item If $i$ is not in the range of $\sigma$, then it means that $M_i = \iota(M_{m+1})$, and then $\partial M_{m+1} \cap M_i = \emptyset$.
\item If $i = \sigma(k)$ for $k \leq k_0$, then we get $\partial M_{m+1} \cap M_i = \partial M_{m+1} \cap M_{i_0}$ by Lemma \ref{lem:trace_sur_le_bord}.
\item If $i = \sigma(k)$ for $k > k_0$, then we get $\partial M_{m+1} \cap M_i = \partial M_{m+1} \cap M_{j_0}$ by Lemma \ref{lem:trace_sur_le_bord}.
\end{enumerate}
In all cases, we have $M_i \cap \partial M_{m+1} \subset M_{i_0} \cup M_{j_0}$. So, $M_i \subset M_{i_0} \cup M_{j_0}$, and the claim is proved.
\end{proof}

From Lemma \ref{lem:assiettes_emboitees}, by injectivity of $\D$ in restriction to $V_1 \cup \cdots \cup V_m$, it follows that $V_1 \cup \cdots \cup V_m = V_{i_0} \cup V_{j_0}$. So, $V_{i_0} \cap V_{j_0} \neq \emptyset$, $V_{m+1} \cap (V_{i_0} \cup V_{j_0}) \neq \emptyset$. Let $\epsilon > 0$ and $V_{m+1}^{s,\epsilon} \supset \bar{V_{m+1}}$ be an open neighborhood of $\bar{V_{m+1}}$ in restriction to which $\D$ is injective and that develops onto $M_{m+1}^{s,\epsilon}$. 

By Lemma \ref{lem:intersection_thickening_connected}, $\D(V_{m+1}^{s,\epsilon}) \cap \D(V_{i_0})$ and $\D(V_{m+1}^{s,\epsilon}) \cap \D(V_{j_0})$ are connected. Let us prove that they intersect. Let $s' := s \circ \iota$. Then, we have $s'(\iota(M_{m+1}) \cap M_{i_0}) = H_{-v,-\alpha_{k_0}}$ and $s'(\iota(M_{m+1}) \cap M_{j_0}) = H_{v,\alpha_{k_0+1}}$ according to Remark \ref{rem:trace_sur_lantipodal}. Since $\alpha_{k_0} > \alpha_{k_0+1}$, $H_{-v,-\alpha_{k_0}} \cap H_{v,\alpha_{k_0+1}} = \{w \in \R^{p,q} \ : \  \alpha_{k_0} > b(w,v) > \alpha_{k_0+1}\}$ is a non-empty strip. This shows that $s'(\iota(M_{m+1}) \cap M_{i_0} \cap M_{j_0})$ contains vectors with arbitrary large Euclidean norm, so $M_{m+1}^{s,\epsilon} \cap M_{i_0} \cap M_{j_0} \neq \emptyset$. Consequently, $\D(V_{m+1}^{s,\epsilon}) \cap (\D(V_{i_0}) \cup \D(V_{j_0}))$ is connected, and Lemma \ref{lem:intersection_connexe} implies that $\D$ is injective in restriction to $V_{m+1}^{\epsilon} \cup V_{i_0} \cup V_{j_0}$. Finally, since $M_{m+1}^{\epsilon} \cup M_{i_0} \cup M_{j_0} = \tilde{\Ein}^{p,q}$, we obtain that $\D$ is a diffeomorphism onto $\tilde{\Ein}^{p,q}$ similarly as before.

\subsubsection{Case of space/time open sets}
  
We finally assume that if we remove the eventual $U_i$ which is empty, the remaining ones are in the second configuration of Lemma \ref{lem:configuration_conforme}. Thus, there is $v \in \R^{p,q}$ such that for all $i$, either $U_i = \emptyset$, or $U_i = v + U_S$, or $U_i = v + U_T$.

Necessarily, $U_1 = \emptyset$ or $U_2 = \emptyset$. Indeed, if both are non-empty, then up to a permutation, $U_1 = v+U_S$ and $U_2 = v+U_T$, for $v \in \R^{p,q}$. It implies that $M_1 = \iota(M_2)$ and in particular $M_1 \cap M_2 = \emptyset$, contradicting $V_1 \cap V_2 \neq \emptyset$. So, $m=3$ and exchanging $V_1$ and $V_2$ if necessary, we may assume $U_1 = \emptyset$ and $U_2 = v + U_S$ or $U_2 = v + U_T$. Let us assume $U_2 = v + U_S$, the other case being similar. The $U_i$'s being pairwise distinct, we must have $U_3 = v+U_T$, implying as above that $M_3 = \iota(M_2)$.

Finally, exchanging $V_1$ and $V_2$ if necessary, we have $m=3$ and $M_1 = \iota(M_4)$, $M_2$ such that $s(M_2 \cap M_4) = v+U_S$ and $M_3 = \iota(M_2)$.

\vspace{.2cm}

Thus, the same reasoning as in the case $\alpha_{k_0} = \alpha_{k_0+1}$ of Section \ref{sss:family_halfspaces} applies if $V_2,V_3$ play the role of $V_{i_0},V_{j_0}$: both are included in a connected injectivity domain of $\D$, and they develop to antipodal Minkowski patches. We thus obtain that $\D$ is also a diffeomorphism in this last situation, completing the proof of Theorem \ref{thm:atlas}.

\section{Projective flatness}
\label{s:projective_flatness}

In this last section, we prove as announced the following proposition.

\begin{proposition}
\label{prop:projective_flat}
Let $\Gamma$ be a cocompact lattice in a connected simple Lie group $G$ of $\R$-rank $n \geq 2$, and let $(M^n,\nabla)$ be a closed $n$-manifold endowed with a linear connection. Let $\alpha : \Gamma \rightarrow \Proj(M,\nabla)$ be a projective action. If $\alpha(\Gamma)$ is infinite, then $\nabla$ is projectively flat.
\end{proposition}

Throughout this section, $\X = \R P^n$ and $\g_{\X} = \sl(n+1,\R)$.

\subsection{Associated Cartan geometry modeled on $\R P^n$}
\label{ss:projective_flatness:cartan_geometry}

We note $P<\PGL(n+1,\R)$ the stabilizer of a line.

\begin{theorem*}[\cite{kobayashi_nagano}]
Let $(M^n,[\nabla])$ be a manifold with a projective class of linear connections. There exist a $P$-principal bundle $\pi_B : B \rightarrow M$ and a $1$-form $\omega \in \Omega^1(B,\g_{\X})$ satisfying the following properties:
\begin{enumerate}
\item for all $b \in B$, $\omega_b : T_bB \rightarrow \g_{\X}$ is a linear isomorphism,
\item for all $A \in \p$, $\omega(A^*) = A$,
\item for all $p \in P$, $(R_p)^* \omega = \Ad(p^{-1}) \omega$,
\end{enumerate}
where $R_p$ stands for the right action of $p$ on $B$ and $A^*$ denotes the fundamental vertical vector field associated to $A$, and such that $\Proj(M,[\nabla])$ is exactly the set of diffeomorphisms $f : M \rightarrow M$ that can be lifted to bundle morphisms $F : B \rightarrow B$ satisfying $F^* \omega = \omega$.
\end{theorem*}

The triple $(M,B,\omega)$ is called the Cartan geometry associated to $(M,[\nabla])$, $\pi_B : B \rightarrow M$ its Cartan bundle and $\omega$ its Cartan connection. The first property implies that the action of $\Proj(M,[\nabla])$ on $B$ is free, and its Lie group structure is - by definition - such that its action on $B$ is moreover proper.

\subsection{Uniform Lyapunov spectrum}
\label{ss:projective_flatness:uniform_lyapunov}

We reuse some of the notations of \cite{BFH}, which we recalled in Section 2.1 of \cite{pecastaing_conformal_lattice}. We note $M^{\alpha} \rightarrow G/\Gamma$ the suspension fiber bundle. We fix $A < G$ a Cartan subspace. Let $\mu$ be any $A$-invariant $A$-ergodic measure on $M^{\alpha}$, which projects to the Haar measure on $G/\Gamma$. Let $\chi_1,\ldots,\chi_r \in \a^*$ be its Lyapunov functionals. Similarly to Proposition 4.1 of \cite{pecastaing_conformal_lattice}, we have:

\begin{lemma}
Such a measure $\mu$ cannot be $G$-invariant.
\end{lemma}

\begin{proof}
Let us assume to the contrary that $\mu$ is $G$-invariant. Then, we get a $\Gamma$-invariant finite measure $\nu$ on $M$. Considering the action of $\Gamma$ on the Cartan bundle $B \rightarrow M$ associated to $[\nabla]$, super-rigidity implies that the cocycle $\Gamma \times M \rightarrow P$ is measurably cohomologous to a compact valued cocyle, as there is no non-trivial homomorphism $\g \rightarrow \mathfrak{gl}(n,\R) \ltimes \R^n$. By the same arguments as in the proof of Lemma 4.4 of \cite{pecastaing_conformal_lattice}, this implies that there exists a finite $\Gamma$-invariant measure $\nu_B$ on $B$. Since $\Proj(M,\nabla)$ acts freely and properly on $B$, it follows that the action $\alpha : \Gamma \rightarrow \Proj(M,\nabla)$ has relatively compact image (see Lemma 4.3 of \cite{pecastaing_conformal_lattice}). In particular, the action preserves a Riemannian metric on $M$, implying that $\alpha$ takes values in a compact Lie group of dimension at most $n(n+1)/2$, hence that $\alpha(\Gamma)$ is finite (see Section 7 of \cite{BFH}), a contradiction.
\end{proof}

Since $\dim M=n=\Rk_{\R}G$, we have $r \leq n$. On the other hand, if $\chi_1,\ldots,\chi_r$ spanned a space of dimension strictly less than $n$, we would get a direction $X \in \a$ on which all the $\chi_i$'s vanish. By Proposition 4.7 of \cite{pecastaing_conformal_lattice} - which is a citation of a central property of the work of \cite{BFH} -, it would imply that $\mu$ is $G$-invariant, a contradiction.

Thus, $r=n$ and $\chi_1,\ldots,\chi_n$ are linearly independent. So, they define a line in $\a$ in restriction to which they all coincide, and similarly to Section 6.2 of \cite{pecastaing_conformal_lattice}, there exists $X \in \a$ such that $\chi_1(X)=\cdots=\chi_n(X)=-1$. The proof of Proposition 6.1 of \cite{pecastaing_conformal_lattice} applies - no conformal geometry is involved in this proposition - and we obtain $g \in G$ and $x \in M$ such that $[(g,x)] \in \Supp \mu$, a sequence $(\gamma_k)$ in $\Gamma$, $(T_k) \rightarrow \infty$ and an open neighborhood $U$ of $x$ such that 
\begin{enumerate}
\item $\gamma_k U \rightarrow \{x\}$ for the Hausdorff topology,
\item $\frac{1}{T_k} \log | D_x \gamma_k v| \rightarrow -1$ for all non-zero $v \in T_xM$
\item $\frac{1}{T_k} \log |\det \text{Jac}_x \gamma_k| \rightarrow -n$.
\end{enumerate}

\subsection{Holonomy sequences associated to $\gamma_k$}
\label{ss:projective_flatness:holonomy_sequences}

Let $\pi_B : B \rightarrow M$ be the Cartan bundle corresponding to $[\nabla]$, with structural group $P \simeq \GL(n,\R) \ltimes \R^n$. Let $A_{\X} < P$ be the Cartan subspace formed of diagonal matrices with positive entries.

\begin{proposition}
Reducing $U$ if necessary, there is a sequence $(a_k)$ in $A_{\X}$ such that for all $y \in U$, there exists a bounded sequence $b_k \in \pi^{-1}(y)$ such that the sequence $\gamma_k b_k a_k^{-1}$ is bounded. Moreover, if $A_k \in \a_{\X}$ is such that $a_k = \exp(A_k)$, we have
\begin{equation*}
\frac{1}{T_k} A_k \rightarrow \diag \left (\frac{n}{n+1},-\frac{1}{n+1},\ldots,-\frac{1}{n+1}\right ).
\end{equation*}
\end{proposition}

\begin{proof}
As $\gamma_k x \rightarrow x$, if $b \in \pi_B^{-1}(x)$, we can choose $p_k' \in P$ such that $\gamma_k b p_k'^{-1}$ is bounded (a \textit{holonomy sequence} for $\gamma_k$ at $b$ in the terminology introduced by Frances). If we decompose $p_k'$ according to $P = G_0 \ltimes \exp(\p^+)$ and if we use the Cartan decomposition of $G_0$, we can write $p_k'  = l_k a_k l_k' \tau_k$, with $a_k \in A_{\X}$, $l_k,l_k' \in G_0$ bounded and $\tau_k \in \exp(\p^+)$. So, if $b_k := bl_k'^{-1}$ and if $\tau_k$ is replaced by $l_k' \tau_k l_k'^{-1} \in \exp(\p^+)$, the we get that $\gamma_k b_k (a_k \tau_k)^{-1}$ is bounded, with $b_k \in \pi^{-1}(x)$ bounded. Let us note $p_k = a_k \tau_k$.

Let $\rho : P \rightarrow \GL(\g_{\X} / \p)$ be the representation induced by the adjoint map. Similarly to Lemma 6.11 of \cite{pecastaing_conformal_lattice}, we have 

\begin{lemma}
\label{lem:lyapunov_holonomy}
Let $(f_k)$ be a sequence of projective maps $(M,[\nabla])$ and $x \in M$ such that $(f_k(x)) \rightarrow x_{\infty}$. The following are equivalent.
\begin{enumerate}
\item $(f_k)$ is Lyapunov regular at $x$, with Lyapunov exponents $\chi_i$ of multiplicity $d_i$.
\item For any $b$ in the fiber of $x$ and any sequence $(p_k)$ in $P$ such that $f_k(b).p_k^{-1} \rightarrow b_{\infty}$, for some $b_{\infty}$ in the fiber of $x_{\infty}$, the sequence $\rho(p_k)$ is Lyapunov regular with Lyapunov exponents $\chi_i$ and multiplicity $d_i$.
\end{enumerate}
\end{lemma}

In our situation, $\gamma_k$ is Lyapunov regular at $x$ with a non-zero Lyapunov exponent of multiplicity $n$. Since $\gamma_k . b . (p_k l_k')^{-1}$ is bounded by construction, up to an extraction, it follows that $\rho(p_k l_k')$ is a Lyapunov regular sequence with a non-zero Lyapunov exponent of multiplicity $n$. From Lemma 6.10 of \cite{pecastaing_conformal_lattice}, we deduce that $\rho(p_k)$ has the same property. Moreover, since $\exp(\p^+)$ is in the kernel of $\rho$, if we note 
\begin{equation*}
a_k =
\begin{pmatrix}
\lambda_0^{(k)} & & \\
& \ddots & \\
& & \lambda_n^{(k)}
\end{pmatrix},
\end{equation*}
we get that $\rho(p_k) = \rho(a_k)$ is conjugate to the diagonal matrix $\diag(\lambda_1^{(k)} {\lambda_0^{(k)}}^{-1},\ldots,\lambda_n^{(k)} {\lambda_0^{(k)}}^{-1})$. If $A_k \in \a_{\X}$ is such that $a_k = \exp(A_k)$, the property of $\rho(p_k)$ means
\begin{equation*}
\frac{1}{T_k} A_k \rightarrow \diag \left (\frac{n}{n+1},-\frac{1}{n+1},\ldots,-\frac{1}{n+1}\right ).
\end{equation*}
We claim now that $(\tau_k)$ is bounded. This can be observed by adapting almost directly the proof of Fait 4.4 of \cite{frances_degenerescence}. 

Indeed, let us write $\tau_k  = \exp(T_k)$ with $T_k = (T_1^{(k)},\ldots,T_n^{(k)})$ (see Section \ref{ss:uniformly_LR_on_X}), and assume to the contrary that a sequence $(T_i^{(k)})$ is unbounded. Up to an extraction, we may assume that $|T_i^{(k)}| \rightarrow \infty$. Then, $p_k$ preserves the projectivization of $\Span(e_1,e_i)$ in $\R P^n$, and acts on it via the matrix
\begin{equation*}
\begin{pmatrix}
\lambda_0^{(k)} & \lambda_0^{(k)} T_i^{(k)} \\
0 & \lambda_i^{(k)}
\end{pmatrix}
\end{equation*}
Then, the same argumentation as in page 17 of \cite{frances_degenerescence} applies literally and gives a sequence of points $x_k \rightarrow x$ such that $\gamma_k x_k \rightarrow y \neq x$, contradicting the fact that $\gamma_k U \rightarrow \{x\}$ for the Hausdorff topology.

So, $(\tau_k)$ is bounded and consequently, if we replace $b_k$ by $b_k \tau_k^{-1}$ which is still bounded, the announced property is valid at $x$ with this choice of $a_k$. Let $\mathcal{U} \subset \n_-$ be a neighborhood of the origin on which the exponential map of the Cartan geometry (see \cite{sharpe}, Ch. 5) is defined at every $b_k$, which exists because $\{b_k\}$ is a relatively compact subset of the fiber $\pi^{-1}(x)$. Given the asymptotic properties of $\Ad(a_k)|_{\n_-}$, we may assume that $\Ad(a_k)$ preserves $\mathcal{U}$, and we have
\begin{equation*}
\forall X \in \mathcal{U}, \ \gamma_k \exp(b_k,X) a_k^{-1} = \exp(\gamma_k b_k a_k^{-1}, \Ad(a_k) X).
\end{equation*}
The fact that $\{b_k\}$ is relatively compact implies that $\cap_{k\geq 0} \pi_B(\exp(b_k,\mathcal{U}))$ is a neighborhood of $x$. If $y$ is in this neighborhood, there is $X_k \in \mathcal{U}$ such that $\pi_B(\exp(b_k,X_k)) = y$ and the formula above implies that $\exp(b_k,X_k)$ is a convenient sequence for $y$ since $\Ad(a_k)X_k$ goes to $0$.
\end{proof}

\subsection{Vanishing of the curvature map near $x$}
\label{ss:projective_flatness:vanishing_curvature}
 
We refer to \cite{sharpe}, Definition 3.22, Ch. 5, for the definition of the curvature map of a Cartan geometry. We note it $\kappa : B \rightarrow \Hom(\Lambda^2 (\g_{\X} / \p),\g_{\X})$. It is $\Aut(M,B,\omega)$-invariant and $P$-equivariant for the right action of $P$ on $\Hom(\Lambda^2 (\g_{\X} / \p),\g_{\X})$ given by $(p.w)(u,v) = \Ad(p^{-1})w(\Ad(p)u,\Ad(p)v)$ for all $w \in \Hom(\Lambda^2 (\g_{\X} / \p),\g_{\X})$ and $u,v \in \g_{\X}/\p$. In particular, if $\kappa$ vanishes at one point $b \in B$, then it vanishes on all of the fiber $b.P$.

Let $y \in U$ and $b_k \in \pi_B^{-1}(y)$ a bounded sequence such that $b_k' := \gamma_k b_k a_k^{-1}$ is bounded. We prove by contradiction that $\kappa$ vanishes in restriction to $\pi_B^{-1}(y)$. So, we assume that it is non-zero at every point of this fiber. Up to an extraction, $b_k \rightarrow b_{\infty}$, and in a basis of $\g_{\X}/\p$ that diagonalizes $\Ad(a_k)$, we pick two vectors $u_i,u_j$ such that $\kappa(b_{\infty})(u_i,u_j) \neq 0$. By equivariance, we have
\begin{align*}
\Ad(a_k)^{-1} \kappa(b_k')(u_i,u_j) & = \kappa(b_k)(\Ad(a_k)^{-1}u_i,\Ad(a_k)^{-1}u_j) \\
 & = {\lambda_0^{(k)}}^2 {\lambda_i^{(k)}}^{-1} {\lambda_j^{(k)}}^{-1} \kappa(b_k)(u_i,u_j).
\end{align*}
This proves that 
\begin{equation*}
\frac{1}{T_k} \log |\Ad(a_k)^{-1} \kappa(b_k')(u_i,u_j)| \rightarrow 2.
\end{equation*}
This is a contradiction because $\kappa(b_k')(u_i,u_j)$ is a bounded sequence in $\g_{\X}$ and for all $\epsilon >0$, $\Ad(a_k)$ acts diagonally on $\g_{\X}$ with all its eigenvalues of modulus at most $e^{(1+\epsilon)T_k}$ for $k$ large enough.

\vspace{.2cm}

This proves that $\kappa$ vanishes on all of $\pi_B^{-1}(U)$.

\subsection{Conclusion}
\label{ss:projective_flatness:conclusion}

We have proved that $\kappa$ vanishes near every point $b$ that projects to a point $x \in M$ such that there exists $g \in G$ such that $[(g,x)] \in \Supp \mu$, for any $A$-invariant, $A$-ergodic finite measure $\mu$ on $M^{\alpha}$ that projects to the Haar measure of $G/\Gamma$.

Similarly to Section 6.6 of \cite{pecastaing_conformal_lattice}, we deduce that for all $\Gamma$-invariant compact $K \subset M$, there is $b \in \pi_B^{-1}(K)$ such that $\kappa$ vanishes on a neighborhood of $b$. Applying this to any orbit closure $\overline{\Gamma.x} \subset M$, we obtain that $\kappa$ vanishes on all of $B$, whence $(M,[\nabla])$ is projectively flat.

\bibliographystyle{alpha}
\bibliography{bibli}

\end{document}